\documentclass[12pt]{amsart}
\usepackage{amssymb}
\usepackage{amsmath,amscd}
\usepackage{amsfonts,latexsym,verbatim,amscd,mathrsfs,color,array}
\usepackage[all,cmtip]{xy}
\usepackage{mathrsfs}
\usepackage{multirow}
\usepackage{graphicx}
\usepackage{amsfonts, amsthm}
\usepackage{bbm}
\usepackage[hmargin=1in,vmargin=1in]{geometry}
\usepackage{mathdots}

\def\Xint#1{\mathchoice
{\XXint\displaystyle\textstyle{#1}}
{\XXint\textstyle\scriptstyle{#1}}
{\XXint\scriptstyle\scriptscriptstyle{#1}}
{\XXint\scriptscriptstyle\scriptscriptstyle{#1}}
\!\int}
\def\XXint#1#2#3{{\setbox0=\hbox{$#1{#2#3}{\int}$ }
\vcenter{\hbox{$#2#3$ }}\kern-.6\wd0}}

\def\dashint{\Xint-}

\newcommand{\LB}{\left[}
\newcommand{\RB}{\right]}
\newcommand{\LA}{\langle}
\newcommand{\RA}{\rangle}

\newcommand{\R}{{\mathbb R}}

\newcommand{\gothg}{{\mathfrak g}}

\newcommand{\scrA}{{\mathscr A}}

\newcommand{\scrG}{{\mathscr G}}

\newcommand{\End}{{\text{End}}}

\newtheorem{thm}{Theorem}[section]
\newtheorem{lemma}[thm]{Lemma}
\newtheorem{prop}[thm]{Proposition}
\newtheorem{cor}[thm]{Corollary}

   \newtheoremstyle{others}
     {3pt}
     {2pt}
     {}
     {}
     {\bf}
     {.}
     {.5em}
     {}
     
\theoremstyle{others}
\newtheorem{rmk}[thm]{Remark}

\numberwithin{equation}{section}

\begin{document}

\title[Instantons and singularities]{Instantons and singularities in the Yang-Mills flow}
\author{Alex Waldron}

\begin{abstract} Several results on existence 
and convergence of the Yang-Mills flow in dimension four are given. We show that a singularity modeled on an instanton cannot form within finite time. Given low initial self-dual energy, we then study convergence of the flow at infinite time. If an Uhlenbeck limit is anti-self-dual and has vanishing self-dual second cohomology, then no bubbling occurs and the flow converges exponentially. We also recover Taubes's existence theorem, and prove asymptotic stability in the appropriate sense.
\end{abstract}

\maketitle

\thispagestyle{empty}

\section*{Introduction.}\label{introduction}

Let $E$ be a vector bundle, with connection $A,$ over a Riemannian base manifold $M.$ Write $F_A$ for the curvature form, $|F_A|^2$ for its pointwise norm in a fixed metric, and $D_A^*$ for the adjoint of the covariant differential. The \emph{Yang-Mills flow}
$$\frac{\partial A}{\partial t} = - D_A^* F_A$$
evolves the connection by the negative gradient of the Yang-Mills functional
\begin{equation*}
\text{YM}(A) = \frac{1}{2} \int_M |F_A|^2 \, dV.
\end{equation*}

The Yang-Mills flow first appeared in the work of Atiyah and Bott \cite{atiyahbott}. It was subsequently shown by G. Daskalopoulos \cite{dask} over compact manifolds of dimension two, and by Rade \cite{rade} in dimensions two and three, that the flow exists for all time and converges. Finite-time blowup is known to occur in dimension five or higher \cite{naito}, and explicit examples of Type-I shrinking solitons were produced on $\mathbb{R}^n, 5\leq n \leq 9,$ by Weinkove \cite{weinkove}. Hong and Tian \cite{hongtian} showed that the singular set has codimension at least four, and gave a complex-analytic description in the compact Kahler case, where an application of the maximum principle shows that singularities can form only at infinite time (\cite{siu}, Ch. 1). In the case of Kahler surfaces, Donaldson's early results \cite{donsurface} for the flow on stable holomorphic bundles have 
been generalized to unstable bundles by Daskalopoulos and Wentworth (\cite{daskwent04}, \cite{daskwent07}).

The behavior of the flow on general Riemannian manifolds of dimension four, however, has not been understood well. Following the analogy with harmonic map flow in dimension two \cite{struwehm}, the foundational work of Struwe \cite{struwe} gives a global weak solution, not excluding the possibility that point singularities (bubbles) will form within finite time. Previously, outside of the Kahler setting, long-time existence and convergence have only been established 
by appealing to energy restrictions on blowup limits \cite{schlatterglobal} or by imposing a symmetric Ansatz \cite{sstz}. Moreover, finite-time singularities have long been known as a characteristic feature of critical harmonic map flow \cite{cdy}.

This paper provides several new 
theorems concerning long-time existence (Theorem \ref{selfdual}, p. \pageref{selfdual}), smooth convergence (Theorem \ref{convergence}, p. \pageref{convergence}), and asymptotic stability (Theorems \ref{parataubes}-\ref{asympstab}, pp. \pageref{parataubes}-\pageref{asympstab}) of the Yang-Mills flow in dimension four. These results rely on the splitting of two-forms into self-dual and anti-self-dual parts, together with a number of small but useful observations in the parabolic setting. 
A thorough background section is included (pp. \pageref{prelims} - \pageref{struwe}). 

\subsection*{Note on dependence of constants} Our estimates will involve 
the following constants.

\vspace{1mm}

\noindent $C_{\ref{dstarf}}$---constant associated to a particular estimate, e.g. Proposition \ref{dstarf}, where its dependence is stated.

\vspace{1mm}

\noindent $C$---universal constant; except during the proof, e.g., of Proposition \ref{dstarf}, where $C = C_{\ref{dstarf}}.$

\vspace{1mm}

\noindent $C_S = 1 + C_{\R^4}$---Sobolev constant appearing in Section \ref{curvestimates}.

\vspace{1mm}

\noindent $C_M$---constant depending only on the geometry of $M.$

\vspace{1mm}

\noindent $C_A$---Poincar\'e constant for a particular connection $A$ obeying (\ref{poincare}).

\vspace{1mm}

\noindent $R_0$---radius, depending on the geometry of $M,$ such that the metric on any geodesic ball of radius less than $R_0$ is sufficiently close to Euclidean, in a sense to be determined.

\vspace{1mm}

\noindent Each of these may increase appropriately, from an earlier to a later appearance; with the exception of $R_0,$ which may decrease, and $C_S,$ which is fixed.

\section{Preliminaries}\label{prelims}


\subsection{Vector bundles and gauge transformations}\label{vectorbundles} 
Let $\pi : E \to M$ be a vector bundle of rank $n,$ with fiberwise inner-product $\LA \cdot, \cdot \RA,$ over a compact, oriented Riemannian base manifold.

A \emph{section} of $E$ over an open set $U \subset M$ is a smooth map $s : U \to E$ such that
$$\pi \circ s = Id_U.$$
By definition, there exists a system of coordinate charts $\{ U^a \}$ for $M,$ together with a \emph{local frame} of sections $\{ e^a_\alpha \}_{\alpha=1}^n$ over $U^a$ for each $a,$ such that any section can be uniquely written
\begin{equation}\label{components}
\left. s \right|_{U^a\cap U} = (s^a)^\alpha e^a_\alpha
\end{equation}
(summing on $\alpha,$ not on $a$). The functions $(s^a)^\alpha$ are referred to as the \emph{local components} of $s.$ 

Applying (\ref{components}) with $U = U^b$, we may write
$$\left. e^a_\alpha \right|_{U^a \cap U^b} = (u^{ab})^\beta{}_\alpha \, e^b_\beta$$
in order to define the \emph{transition functions} $(u^{ab})^\beta {}_\alpha.$ This yields, for any section $s,$ the familiar transformation law
\begin{equation}\label{transf}
\left(s^b \right)^\beta = \left( u^{ab} \right)^\beta{}_\alpha \left( s^a \right)^\alpha.
\end{equation}
By definition, the transition functions (invertible matrices) satisfy the \emph{cocycle conditions}
$$u^{bc} \cdot u^{ab} = u^{ac}$$
on $U^a \cap U^b \cap U^c.$ Conversely, these data are sufficient to reconstruct the bundle $E.$

Choosing the local frames to be orthonormal
$$\LA e_\alpha^a, e_\beta^a \RA = \delta_{\alpha \beta}$$
ensures that the $u^{ab}$ lie inside the orthogonal group $O(n).$ Should these lie within a subgroup $G \subset O(n),$ we say that $E$ has \emph{structure group} $G.$ Since any compact Lie group $G$ embeds into $O(n)$ for some $n,$ studying vector rather than principal bundles with compact structure group entails no loss of generality.\footnote{It will be clear that if the connection takes values in the Lie algebra $\mathfrak{g}$ of the group $G,$ then this property will be preserved as long as we deal with smooth connections and gauge transformations, and in fact more generally (see \cite{donkron}).}

Henceforth, we will suppress the chart label and local frame, writing $s^\alpha$ for a section of $E$ in local components, with Greek index, and $s_\alpha$ for a section of $E^*.$
A Latin index $v^i$ corresponds to the section $v^i \frac{\partial}{\partial x^i}$ of the tangent bundle $TM,$ and $v_i$ to a section $v_i \, dx^i$ of the cotangent bundle $T^* M.$ 

The set of \emph{gauge transformations} $\scrG_E \subset \End E $ consists of the orthogonal matrices at each point (or elements of the structure group $G$), and a section $u$ of $\left. \scrG_E \right|_U$ defines a local metric-preserving automorphism of $E.$ 
The vector bundle of \emph{infinitesimal gauge transformations} $\gothg_E \subset \End E$ consists of skew-symmetric matrices (or elements of $\gothg$), and the sections of $\left. \gothg_E \right|_U$ correspond to the Lie algebra of $\left. \scrG_E \right|_U$ via exponentiation.
We denote the induced action of $u$ on any tensor by $u(\cdot),$ which on $\gothg_E$ coincides with the adjoint action.

We write $\Omega^k(E)$ for the bundle of \emph{$E$-valued $k$-forms}, or alternating elements of $(T^*M)^{\otimes k} \otimes E,$ with inner-product $\LA \cdot, \cdot \RA$ induced from the standard orthonormal basis of wedge elements. The components of a two-form $\omega,$ for instance, are defined by
$$\omega = \sum_{i<j} \omega_{ij} dx^i \wedge dx^j = \frac{1}{2}\omega_{ij} dx^i \wedge dx^j.$$ Write $\Omega^k (\gothg_E) \subset \Omega^k (\End E)$ for the \emph{Lie-algebra valued $k$-forms}. For $\omega, \eta \in \Omega^2(\End E),$ and similarly for forms of any degree, we define the \emph{wedge product}
$$\left( \omega \wedge \eta \right)^\alpha {}_\beta = \, \frac{1}{4} \omega_{ij}{}^\alpha{}_\gamma \, \eta_{k\ell} {}^\gamma {}_\beta \left( dx^i \wedge dx^j \wedge dx^k \wedge dx^\ell \right).$$
Defining the operator $\ast: \Omega^k \left( \gothg_E \right) \to \Omega^{4-k}(\gothg_E)$ as the linear extension of the ordinary Hodge star on differential forms, we obtain the relation
\begin{equation}\label{hodgestar}
- Tr \, \omega \wedge \ast \eta = \LA \omega, \eta \RA \, dV
\end{equation}
for $\omega, \eta \in \Omega^k(\gothg_E).$

In dimension four, the Hodge star satisfies
$$\ast^2 = (-1)^{k(4-k)}= (-1)^k$$
on $\Omega^k.$ For this reason, the two-forms (valued in any bundle) split into orthogonal positive and negative eigenspaces
$$\Omega^2 = \Omega^{2+} \oplus \Omega^{2-}.$$
A form $\omega \in \Omega^{2\pm}$ which satisfies $\ast \omega = \pm \omega$ is called \emph{self-dual} or \emph{anti-self-dual}, respectively. In normal coordinates at a point, (anti)-self-duality amounts to the three relations
\begin{equation}\label{asd}
\omega_{12} = \pm \omega_{34} \quad \quad \quad
\omega_{13} = \mp \omega_{24} \quad \quad \quad
\omega_{14} = \pm \omega_{23}.
\end{equation}

\subsection{Connections and covariant derivatives} A \emph{connection} $A$  is a metric-preserving rule for transporting fiber elements of $E,$ which is linear in the tangent directions of $M.$

Formally, a connection is equivalent to a \emph{covariant derivative}, or a map
$$
s \mapsto \nabla_A s
$$
from sections of $E$ to sections of $T^*M \otimes E,$ which satisfies
\begin{equation*}\label{leibniz}
\begin{split}
\nabla_A (f\cdot s) & = df \otimes s + f \nabla_A s \\
d \LA s,t \RA & = \LA \nabla_A s , t \RA + \LA s, \nabla_A t \RA
\end{split}
\end{equation*}
for any smooth function $f.$ In local coordinates, writing $\left( \nabla_A s \right)\left( \partial_i \right) = \nabla_i s,$ we may define the connection components
$$A_{i \beta}^\alpha = \LA e_\alpha , \nabla_i e_\beta \RA$$
in order to obtain the well-known formula
$$\nabla_i s^\alpha := (\nabla_A s)_i{}^\alpha = \partial_i s^\alpha + A_{i \beta} ^\alpha s^\beta.
$$

Under a gauge transformation or change-of-frame $u,$ the components of $A$ must transform according to the requirement
$$u(\nabla_A s) = \nabla_{u(A)} (u(s))$$
or in matrix notation
\begin{equation}\label{conntransf}
u(A) = u \cdot A \cdot u^{-1} - du \cdot u^{-1}.
\end{equation}
From this transformation law, it is evident that the difference of any two connections defines a genuine section of $\Omega^1(\gothg_E),$ as does the derivative $\dot{A}$ of a smooth family of connections. The set of all connections is thus an affine space $\scrA_E$ modeled on $\Omega^1(\gothg_E).$

Define the \emph{covariant differential} on sections $\Omega^k(E) \to \Omega^{k+1}(E)$ by the rule
\begin{equation*}
D_A(s^\alpha dx^{i_1} \wedge \cdots \wedge dx^{i_k}) = \nabla_i s^\alpha dx^i \wedge dx^{i_1} \wedge \cdots \wedge dx^{i_k}.
\end{equation*}
By abuse of notation, 
we may consider $A$ in local coordinates as a $\gothg$-valued ``\emph{connection 1-form},'' $A_{i\beta}^\alpha dx^i,$ and rewrite $D_A$ in terms of the wedge product, as follows. For $\alpha \in \Omega^k(E),$ we will write
\begin{equation*}
D_A \alpha = d\alpha + A \wedge \alpha
\end{equation*}
and for $\omega \in \Omega^k(\End E)$
\begin{equation}\label{kplusone}
D_A \omega = d\omega + A \wedge \omega + (-1)^{k+1} \omega \wedge A.
\end{equation}
We will choose based on efficiency whether to employ the form or the index notation in each derivation that follows. 
The \emph{adjoint} of the covariant derivative is given by
\begin{equation}\label{adjoint}
(\nabla_A^* \omega)_{i_1 \cdots i_k} = - g^{\ell j} \nabla_\ell \omega_{j i_1 \cdots i_k} = - \nabla^j \omega_{j i_1 \cdots i_k}
\end{equation}
which agrees, on form components, with the adjoint of the covariant differential
$$D_A^* = - \ast D_A \ast.$$

\subsection{Curvature and Bianchi identities} The \emph{curvature} $F_A$ of the connection $A$ is defined as the operator on sections of $E$
\begin{equation*}
\begin{split}
\left( D_A \right)^2 s & = D_A ( ds + A \cdot s) \\
& = d^2s + dA \cdot s - A \wedge ds + A \wedge ds + A \wedge A \cdot s \\
& = (dA + A \wedge A) s.
\end{split}
\end{equation*}
This operator is evidently $C^\infty$-linear, and defines a section
$$\frac{1}{2} F_{ij}dx^i \wedge dx^j \in \Omega^2(\gothg_E)$$ with components
$$F_{ij}{}^\alpha{}_\beta = \partial_i A_{j \beta}^\alpha - \partial_j A_{i \beta}^\alpha + A_{i \gamma}^\alpha A_{j \beta}^\gamma - A_{j \gamma}^\alpha A_{i \beta}^\gamma.$$

Writing $R_{ij}{}^k{}_\ell$ for the curvature of $\Gamma$ on $TM,$ we obtain the commutation formula
\begin{equation}\label{commie}
\begin{split}
\LB \nabla_i , \nabla_j \RB t^k{}_\ell{}^\alpha{}_\beta & = R_{ij}{}^k{}_m t^m{}_\ell{}^\alpha{}_\beta - R_{ij}{}^n{}_\ell t^k{}_n{}^\alpha{}_\beta \\
& \quad + F_{ij}{}^\alpha{}_\gamma t^k{}_j{}^\gamma{}_\beta - F_{ij}{}^\gamma{}_\beta t^k{}_j{}^\alpha{}_\gamma
\end{split}
\end{equation}
and similar formulae in general.
Note the identity
\begin{equation}\label{firstbianchi}
\begin{split}
\left( D_A^* \right)^2 F_A & = \nabla^i \nabla^j F_{ij}  = \frac{1}{2}(\nabla^i \nabla^j - \nabla^j \nabla^i)F_{ij} \\
& = \frac{1}{2}\left(-R^{ijn}{}_i F_{nj} - R^{ijn}{}_j F_{in} +  \LB F^{ij} , F_{ij} \RB \right) \\
& = 0.
\end{split}
\end{equation}
We have also the \emph{second Bianchi identity}
\begin{equation*}
\begin{split}
D_A F_A & = d(dA + A \wedge A) + A \wedge dA - dA \wedge A + A \wedge (A \wedge A) - (A\wedge A) \wedge A \\
& = d^2A + dA \wedge A - A \wedge dA + A \wedge dA - dA \wedge A \\
& = 0.
\end{split}
\end{equation*}
The latter is equivalent to the familiar identity on component matrices
$$\nabla_i F_{jk} + \nabla_j F_{ki} + \nabla_k F_{ij} = 0.$$

\subsection{Yang-Mills and instantons}\label{ymfunctional} The $L^2$ gradient of the Yang-Mills energy is obtained as follows.

For a one-form $a \in \Omega^1(\gothg_E),$ note the formula
\begin{equation}\label{curvderiv}
\begin{split}
F_{A+a} & = F_A + da + A \wedge a + a \wedge A + a \wedge a \\
& = F_A + D_A a + a \wedge a
\end{split}
\end{equation} 
and compute
\begin{equation*}
\begin{split}
\left. \frac{d}{dt}YM(A + ta) \right|_{t=0} & = \frac{1}{2} \frac{d}{dt} \left( \int \left(|F_A|^2 + 2 t \, \LA F_A, D_A a \RA \right) dV + O(t^2) \right) \\
& = \int \LA a, D_A^* F_A \RA \, dV.
\end{split}
\end{equation*}
The Yang-Mills flow is therefore defined as above, or in local components
$$\frac{\partial}{\partial t} A_{j \beta} ^\alpha = \nabla^i F_{i j}{}^\alpha{}_\beta.$$ 

By construction, any sufficiently smooth solution will satisfy the \emph{energy inequality}
$$
\text{YM}(A(0)) - \text{YM}(A(T)) = \int_0^T \|D_A^* F_A \|^2 dt.
$$
We may therefore expect a weak limit
$$A(t) \rightharpoonup A_\infty \quad (t \to \infty)$$
which, if not a minimum of $\text{YM},$ is at least a \emph{Yang-Mills connection}, satisfying
$$D_{A_\infty}^* F_{A_\infty} = 0.$$
Note that we will often abbreviate
$$\| \cdot \| = \|\cdot \|_{L^2(M)}.$$


For the self-dual and anti-self-dual parts of the curvature form, write
$$F^{\pm} = \frac{1}{2}(F \pm \ast F)$$
respectively.
From the second Bianchi identity, remark that 
\begin{equation}\label{selfdualbianchi}
\begin{split}
2 D^*F^\pm & = - \ast ( D \ast F \pm D \ast^2 F ) \\
& = D^* F.
\end{split}
\end{equation}
Therefore, if a connection is self-dual ($F^- = 0$) or anti-self-dual ($F^+ = 0$), it is a critical point of the Yang-Mills energy. These special Yang-Mills connections are called \emph{instantons}.

Recall from Chern-Weil theory that the characteristic number
\begin{equation*}
\begin{split}
\kappa(E) = \frac{1}{8\pi^2} \int Tr F \wedge F
\end{split}
\end{equation*}
is a topological invariant of $E,$ which for complex bundles coincides with the second Chern character. From (\ref{hodgestar}), we compute
\begin{equation*}
\begin{split}
\int Tr F \wedge F & = - \int \big\LA F^+ + F^-, F^+ - F^- \big\RA \, dV \\
& = \|F^-\|^2 - \|F^+\|^2
\end{split}
\end{equation*}
but by orthogonality, also
$$\|F\|^2 = \|F^+\|^2 + \|F^-\|^2.$$
We obtain the formula
\begin{equation}\label{energyformula}
\|F\|^2 = 8\pi^2 \kappa + 2 \|F^+\|^2.
\end{equation}
A connection is therefore anti-self-dual if and only if it attains the minimum energy $4\pi^2 \kappa.$

Assuming $\kappa(E) \geq 0,$ without loss of generality, one might expect to find anti-self-dual instantons on $E$---a principal aim in Donaldson theory (\cite{don4d}, \cite{donkron}).

\subsection{Evolution of curvature and Weitzenbock formulae}\label{weitzsection}

From (\ref{curvderiv}), we compute the evolution
$$\frac{\partial}{\partial t} F_A = D_A (- D_A^* F_A).$$ 
In view of the second Bianchi identity, we may rewrite this as the tensorial heat equation
$$\left( \frac{\partial}{\partial t} + \Delta_A \right) F_A = 0$$
where $\Delta_A = D_A D_A^* + D_A^* D_A$ is the Hodge Laplacian with respect to the evolving connection $A = A(t).$ Henceforth we suppress the label and write $D = D_A, \nabla = \nabla_A,$ etc., although for emphasis we continue to denote the Hodge Laplacian by $\Delta_A.$

We compute, for $\omega \in \Omega^k(\gothg_E)$
\begin{equation*}
\begin{split}
(D^*D + D D^*) \omega_{i_1 \cdots i_k} & = - \nabla^j \left(\nabla_j \omega_{i_1 \cdots i_k} - \nabla_{i_1} \omega_{j i_2 \cdots i_k} - \cdots - \nabla_{i_k} \omega_{i_1 \cdots i_{k-1} j} \right) \\
& - \nabla_{i_1} \nabla^j \omega_{j i_2 \cdots i_k} + \nabla_{i_2} \nabla^j \omega_{j i_1 i_3 \cdots i_k} + \cdots + \nabla_{i_k} \nabla^j \omega_{j i_2 \cdots i_{k-1} i_1 }.
\end{split}
\end{equation*} 
Permuting $j$ and $i_1$ in the $+$ terms of the second line, we may group all but the first term of the first line into commutators. We obtain the \emph{Weitzenbock formula}
\begin{equation*}
\left( \Delta_A \omega \right)_{i_1 \cdots i_k} = \nabla^* \nabla \omega_{i_1 \cdots i_k} - \LB F_{i_1}{}^j , \omega_{j i_2 \cdots i_k} \RB - \cdots - \LB F_{i_k}{}^j , \omega_{i_1 \cdots i_{k-1} j} \RB + Rm \# \omega
\end{equation*} 
where $\#$ is a bilinear expression in the tensor components. In particular, for a two-form, we have
\begin{equation}\label{weitz}
\begin{split}
- \left( \Delta_A \omega \right)_{ij} & = \nabla^k \nabla_k \omega_{ij} + \LB F_i{}^k , \omega_{k j} \RB - \LB F_j{}^k , \omega_{k i} \RB \\
& - R_i{}^k{}^\ell{}_k \omega_{\ell j} -  R_i{}^k{}^\ell{}_j \omega_{k \ell} + R_j{}^k{}^\ell{}_k \omega_{\ell i} +  R_j{}^k{}^\ell{}_i \omega_{k \ell}.
\end{split}
\end{equation}

We now make a simple observation about the zeroth-order terms---see \cite{lawson}, Appendix II.
For $\omega \in \Omega^{2+}$ and $\eta \in \Omega^{2-},$ we may write in normal coordinates
\begin{equation*}
\begin{split}
\omega_{1k} \eta_{k 2} - \omega_{2k} \eta_{k 1} & = \omega_{13} \eta_{3 2} - \omega_{23} \eta_{31} + \omega_{14} \eta_{4 2} - \omega_{24} \eta_{41} \\
& = (-\omega_{24}) (- \eta_{41}) - \omega_{14} \eta_{42} + \omega_{14} \eta_{4 2} - \omega_{24} \eta_{41} \\
& = 0
\end{split}
\end{equation*}
and likewise for any choice of indices. A similar calculation shows that for $\omega, \bar{\omega}$ self-dual, $\omega_{i k} \bar{\omega}_{k j} - \omega_{j k} \bar{\omega}_{k i}$ is again self-dual. These facts amount to the well-known splitting of Lie algebras
$$so(4) = so(3) \oplus so(3).$$
For the $Rm $ terms of (\ref{weitz}), one notes that the first and third are skew in $i,j,$ as are the second and fourth, and that these are each self-dual if the same is true of $\omega.$\footnote{See Freed and Uhlenbeck \cite{freeduhl}, Appendix C, for an explanation of the splitting, and for the derivation of the well-known expression ``$Rm \# = W^+ - R/3$'' in (\ref{sdweitz}).}


We conclude that the extra terms of the Weitzenbock formula (\ref{weitz}) in fact split into self-dual and anti-self-dual parts. Note also that $\Delta_A \ast = \ast \Delta_A,$ and the trace Laplacian preserves self-duality since the same is true of the (metric-compatible) 
covariant derivative. Hence
\begin{equation}\label{sdweitz}
- \left(\Delta_A \omega \right)_{ij} = \nabla^k \nabla_k \omega_{ij} + \LB F^+_i{}^k , \omega_{kj} \RB - \LB F^+_j{}^k , \omega_{ki} \RB + Rm \# \omega
\end{equation} 
for $\omega$ self-dual, and a similar formula holds for anti-self-dual forms. Applied to the self-dual curvature $F^+,$ (\ref{sdweitz}) yields the key evolution equation
\begin{equation}\label{fplusevolution}
\frac{\partial}{\partial t} F^+_{ij} = \nabla^k \nabla_k F^+_{ij} + 2 \LB F^+_i{}^k , F^+_{kj} \RB + Rm \# F^+.
\end{equation}

\subsection{Sobolev spaces} Fix, for the remainder of the paper, a smooth reference connection $\nabla_{ref}$ on $E.$
For $\omega \in \Omega^k(\End E)$ and any domain $\Omega \subset M,$ define the Sobolev norms
$$\|\omega\|_{L^p_k(\Omega)} = \left( \sum_{\ell=0}^k \|\nabla_{ref}^{(\ell)} \omega \|_{L^p(\Omega)}^p \right)^{1/p}$$
and also write
$$H^k(\Omega) = L^2_k(\Omega).$$


By (\ref{conntransf}), any connection can 
be uniquely written $D_{ref} + A,$ with $A\in \Omega^1(\gothg_E).$ We may therefore define the Sobolev norm of a connection, $\|A\|_{L^p_k(\Omega)},$ as the norm of this corresponding one-form.\footnote{A different choice of reference connection yields uniformly equivalent norms. It is important to emphasize, however, that the zeroth-order norms of a \emph{difference} of two connections, being a genuine 1-form, does not depend on the choice of gauge (or on $\nabla_{ref}$), and moreover that $\scrG_E$ acts by isometries for the $L^p = L^p_0$ distances on the affine space $\scrA_E.$}
The Sobolev space of $L^p_k(\Omega)$ connections is defined as the completion of $\left. \scrA_E \right|_{\Omega}$ with respect to the $L^p_k(\Omega)$ norm. All connections appearing in this paper will be assumed to be smooth, unless explicitly stated otherwise. 

For two open sets $\Omega_1 \subset \subset \Omega \subset M,$ there is a local Sobolev inequality
\begin{equation}\label{localsobolev}
\|\omega\|_{L^4(\Omega_1)} \leq C_{(\ref{localsobolev})} \|\omega\|_{H^1(\Omega)}.
\end{equation} 
This follows from the ordinary Sobolev inequality in dimension four and the Kato inequality; hence $C_{(\ref{localsobolev})}$ depends on $\Omega_1$ and $\Omega,$ but not $\nabla_{ref}.$ Henceforth, unless otherwise stated, we will also abbreviate $L^p_k = L^p_k (M).$ 

From (\ref{localsobolev}), note that the Yang-Mills energy is controlled locally by the $H^1$ norm of the connection. The converse also holds true 
for connections of small energy, after choosing a local Coulomb gauge \cite{uhlenbecklp}. However, due to the zeroth-order terms of the Weitzenbock formula, the Sobolev constant for $D_A \oplus D_A^*$ may blow up as the curvature of $A$ concentrates.

\subsection{Short-time existence}\label{shorttimesection}


For the benefit of the reader, we briefly recall the construction of a solution to the Yang-Mills flow
$$D(t) = D_{ref} + A(t)$$
on a maximal time interval $0 \leq t < T,$ due in dimension four independently to Struwe \cite{struwe} and to Kozono, Maeda, and Naito \cite{kozono}. In subsequent sections, $A(t)$ will always denote a solution of the form described here, which will in particular be smooth for $0< t <T$ (after a fixed change of gauge). Because the main interest of this paper is not short-time existence theory, but rather long-time existence and convergence, the discussion will be slightly informal at times.


\subsubsection{Smooth initial data} For a sufficiently smooth initial connection, short-time existence follows by a De Turck-type trick first employed by Donaldson \cite{donsurface} (see also \cite{donkron}, \S 6).

Assume first that $D(t)$ is a family of connections depending smoothly on time, $u_t \in \scrG_E$ is a family of gauge transformations with $u_{t_0} = 1,$ 
and define
$$\bar{D}(t) = u_t(D(t)).$$
The transformation law (\ref{conntransf})
and the Leibniz rule give
\begin{equation}\label{gaugederivid}
\left. \frac{d}{dt} \bar{D} \right|_{t=t_0} = \left. \frac{d}{dt} D \right|_{t=t_0} - Ds
\end{equation}
where $s = \left. \frac{d}{dt} u_t \right|_{t=t_0}.$
In general, consider the gauge transformation $u_t \cdot u_{t_0}^{-1}$ in (\ref{gaugederivid}), and apply $u_{t_0}$ to both sides. Letting
$$s(t) = u_t^{-1} \frac{d}{dt} u_t \in \gothg_E$$
this yields
\begin{equation}\label{gaugederiv}
\left. \frac{d}{dt} \bar{D} \right|_{t=t_0} = u_{t_0}\left( \left. \frac{d}{dt} D \right|_{t=t_0} \right) - \bar{D}s(t_0).
\end{equation}


Now fix a smooth connection $D_1.$ To solve the Yang-Mills flow for a smooth initial connection $D_0 = D_1 + A_0,$ write $\bar{D}(t) = D_1 + a(t)$ and consider the alternate equation
\begin{equation}\label{gaugequiv}
\frac{d}{dt} \bar{D} = \frac{d}{dt} a = - \bar{D}^* \bar{F} + \bar{D} \left( -\bar{D}^* a \right), \quad a(0) = A_0.
\end{equation}
Recall that
$$\bar{F} = F_{D_1} + D_1 a + a \wedge a = F_{D_1} + \bar{D} a + a \# a$$
so we may rewrite (\ref{gaugequiv})
\begin{equation*}
\begin{split}
\frac{d}{dt} a & = -\left( \bar{D} \bar{D}^* + \bar{D}^* \bar{D} \right) a - \bar{D}^* \left( F_{D_1} + a \# a \right) \\
& = - \Delta_{D_1} a + P(a, \nabla_1 a).
\end{split}
\end{equation*}
with $P$ a polynomial function of the tensor components with smooth coefficients.

The Sobolev multiplication theorems---see \cite{freeduhl}, Appendix A---imply that for $k \geq k_0$ sufficiently large, $P$ determines a smooth map of Hilbert spaces 
\begin{equation}\label{smoothnonlinearity}
H^k \left(\Omega^1(\gothg_E)\right) \to H^{k-1}\left(\Omega^1(\gothg_E)\right).
\end{equation}
Standard parabolic theory---in particular Lemma 3.2 of Struwe \cite{struwe}, and a straightforward instance of the fixed-point argument used in proving Theorem \ref{localexistence} below---provides a unique solution $a(t)$ for $ 0 \leq t \leq \tau,$ with $\tau$ depending only on $D_1$ and a fixed upper bound for $\|A_0\|_{H^k}.$ This solution is continuous in $H^k$ with respect to time and initial data---see Lions-Magenes \cite{lm} for basic theory.\footnote{See also Rade \cite{rade} for a slightly different method, giving the optimal $k_0$ for which these properties hold.} 

Since $D_0$ is smooth and $k \geq k_0$ was arbitrary, the solution $\bar{D}(t)$ is smooth. We may therefore define a smooth gauge-transformation $u = u_t$ by the pointwise ODE
\begin{equation}\label{gaugeode}
s = u^{-1} \frac{d}{dt} u = \bar{D}^* a, \quad u_0 = 1.
\end{equation} 
By (\ref{gaugequiv}), $D(t) = u^{-1}(\bar{D})$ is a classical solution of the Yang-Mills flow.

This solution can then be extended in the usual way: provided that there exists a smooth limit\footnote{Throughout the paper, a \emph{smooth limit} will mean a limit in $C^\infty,$ i.e. $C^k$ for all $k,$ over the domain $\Omega \subset M$ in question.} 
$D(t) \to D(T)$ as $t \to T$ over all of $M,$ we may continue the flow by concatenating another short-time solution with initial data $D(T).$ Although the intrinsic characterization of the maximal smooth extension time $T$ follows directly from the parabolic estimates of Section \ref{estimatesection} below, it will be useful to recall Struwe's sharp local existence result and to summarize his method.

\subsubsection{Rough initial data---Struwe's approach} For initial data $D_0 \in H^1,$ Struwe \cite{struwe} writes
$$\bar{D}(t) = D_1 + \bar{A}(t) = D_1 + A_{bg}(t) + a(t)$$ 
where $D_1$ is a smooth connection near $D_0,$ and $A_{bg}$ solves the ordinary heat equation with respect to $D_1,$ with
$$A_{bg}(0) = A_0 = D_0 - D_1.$$
The remaining piece $a(t),$ with $a(0)=0,$ is determined by a fixed-point argument based on sharper estimates for the nonlinearity $P$ defined above. 
\begin{thm}\label{localexistence} (Struwe \cite{struwe}, \S 4.2-4.3) Given a smooth connection $D_1,$ there exist $C_{\ref{localexistence}}$ and $\epsilon > 0$ (depending only on $(E, D_{ref})$) 
and $\tau$ (depending on $D_1$) as follows. For any $A_0 \in H^1$ with $\|A_0\|_{H^1} < \epsilon,$ there exists a smooth solution $\bar{D}(t) = D_1 + \bar{A}(t)$ to (\ref{gaugequiv}) for $0 < t \leq \tau,$ with
$$\|\bar{A}(t)\|_{H^1} \leq C_{\ref{localexistence}} \|A_0\|_{H^1}$$
and $\bar{A}(t) \to A_0$ strongly in $H^1$ as $t \to 0.$
\end{thm}


The desired weak solution $D(t)$ of the Yang-Mills flow, in the sense defined by Struwe, 
is then obtained as follows. Fix a time $0 < t_0 < \tau,$
and let $\hat{D}(t)$ be the solution with
$$\hat{D}(t_0) = \bar{D}(t_0)$$
obtained by solving (\ref{gaugeode}) for $0 < t < \tau.$ 
For any sequence of times $t_i \to 0,$ by construction 
there exist smooth gauge transformations $u_i$ such that
$$u_i(\hat{D}(t_i)) = \bar{D}(t_i).$$
Theorem \ref{localexistence} implies
$$u_i(\hat{D}(t_i)) \stackrel{H^1}{\longrightarrow} D_0.$$
Struwe also finds a strong $H^1$ limit $u_i \to u_0,$ 
and defines
$D(t)= u_0(\hat{D})$
as a weak solution with
$$D(t) \stackrel{L^2}{\longrightarrow} D_0$$
as $t \to 0.$
The solution $D(t)$ is smooth for $0 < t < \tau,$ modulo the constant gauge transformation $u_0.$ 

\subsubsection{Energy concentration} Both the constructions of Struwe, and of Kozono et. al. \cite{kozono}, yield the following criterion for long-time existence.
For a certain $\epsilon_0 > 0,$ we say that the curvature $F(t) = F_{A(t)}$ \emph{concentrates} (in $L^2$) at $x \in M$ if
$$\inf_{R > 0} \limsup_{t \to T} \int_{B_R(x)} |F(t)|^2 dV \geq \epsilon_0.$$
\begin{thm}\label{struwe}
The maximal smooth existence time,\footnote{Although the short-time solutions of (\ref{gaugequiv}) are unique, there is no canonical inverse to the construction, and one cannot conclude uniqueness for the Yang-Mills flow.
Kozono et. al. (\cite{kozono}, Theorem C) do prove that solutions in $H^k, k >> 1,$ are unique \emph{modulo gauge}, as expected. Struwe, moreover, proves that as long as $D(t)$ is irreducible ($D(t): \Omega^0(\gothg_E) \to \Omega^1(\gothg_E)$ has trivial kernel), then his weak solution is unique, including gauge.} $T,$ is characterized by concentration of the curvature $F(t)$ at some $x\in M$ as $t\to T.$
\end{thm}


\section{Finite-time existence}\label{selfdualsection} 

In view of Theorem \ref{struwe}, to prove long-time existence of the Yang-Mills flow in dimension four, it remains to control the concentration of curvature.

The first step is to obtain separate control of the self-dual curvature directly from (\ref{fplusevolution}). Applying the inner-product with $F^+$ yields 
\begin{equation*}
\left(\frac{\partial}{\partial t} + \Delta \right) |F^+|^2 = -2 |\nabla F^+|^2 + 2 \LA  F^+{}^{ij}, \LB F^+_i{}^k , F^+_{kj} \RB \RA + Rm \# F^+ \# F^+.
\end{equation*}
Writing $u=|F^+|^2,$ we obtain the differential inequality
\begin{equation}\label{fplusev}
\left(\frac{\partial}{\partial t} + \Delta \right) u \leq A u^{3/2} + B u
\end{equation}
where $A$ is a universal constant, and $B$ is a universal constant times $\|Rm\|_{L^\infty(M)}.$ Similar inequalities hold for the anti-self-dual curvature $F^-,$ and the full curvature $F.$


\subsection{Estimates on curvature evolution}\label{curvestimates} Estimates for (\ref{fplusev}) follow by adapting Moser iteration to the present borderline situation. This can be done in two different ways. Apart from Theorem \ref{parataubes}, later results will only require Proposition \ref{moser}, with $p=1.$

\begin{prop}\label{premoser} Let $u$ be a nonnegative smooth function satisfying (\ref{fplusev}) on $M \times \LB 0, T \right),$ with $M$ compact of dimension four.
There exists a universal constant $\delta > 0,$ and $R_0 > 0,$ depending on the geometry of $M,$ as follows. If $R < R_0$  is such that
$$\int_{B_R(x_0)} u(t) < \delta^2$$
for all $x_0 \in M$ and $0 \leq t < T,$ then
$$\| u(t) \|_{L^2} \leq  e^{-\frac{t}{c R^2}} \, \| u(0) \|_{L^2} + C R^{-2} \left( 1-e^{-\frac{2 \, t}{c R^2}} \right)^{1/2} \sup_{0 \leq s \leq t} \| u(s) \|_{L^1}.$$
\end{prop}
\begin{proof}
Let $\varphi\in C^\infty_0(B_R(x_0)).$ Multiplying (\ref{fplusev}) by $\varphi^2 u$ and integrating by parts, we obtain
\begin{equation*}
\begin{split}
\frac{1}{2} \frac{d}{dt} \left(\int \varphi^2 u^2\right) + \int \nabla(\varphi^2 u) \cdot \nabla u & \leq A\int \varphi^2 u^{5/2} + B \int \varphi^2 u^2 \\
\frac{1}{2} \frac{d}{dt} \left( \int \varphi^2 u^2 \right) + \int |\nabla(\varphi u)|^2  & \leq \int |\nabla \varphi |^2 u^2 + A\int \varphi^2 u^{5/2} + B \int \varphi^2 u^2.
\end{split}
\end{equation*} 
Applying the Sobolev and H\"older inequalities on $B_R$ yields
\begin{equation*}
\frac{1}{2} \frac{d}{dt} \int \varphi^2 u^2 + \left( \frac{1}{C_S} - A\delta \right) \left(\int (\varphi u)^4 \right)^{1/2}  \leq \|\nabla \varphi\|^2_{L^{\infty}} \int_{B_R} u^2  + B \int \varphi^2 u^2.
\end{equation*} 

Provided that $R < R_0,$ depending on the geometry of $M,$ we may assume the following. First, that for any $x \in M,$ the volume of any geodesic ball
$$Vol(B_R(x)) \leq c^2 R^4.$$
Second, that compactly supported functions on $B_R(x)$ obey a Sobolev inequality, with constant $C_S$ independent of $x.$ Third, that it is possible to choose a cover of $M$ by geodesic balls $B_{R/2}(x_i)$ in such a way that no more than $N$ of the balls $B_i=B_R(x_i)$ intersect a fixed ball, with $N$ universal in dimension four. For each $i,$ let $\tilde{\varphi}_i$ be a cutoff 
for $B_{R/2}(x_i) \subset B_R(x_i)$ with $\|\nabla \tilde{\varphi}_i\|_{L^\infty} < 4/R,$ and define
$$\varphi_i = \tilde{\varphi_i}/ \sqrt{\sum \tilde{\varphi}_j^2}.$$
Then $\{ \varphi_i^2 \}$ is a partition of unity for $M,$ with
$$\|\nabla \varphi_i\|_{L^\infty} < C/R.$$

We now apply the above differential inequality with $\varphi = \varphi_i,$ and sum on $i,$ to obtain
\begin{equation*}
\begin{split}
\sum_i \left( \frac{1}{2} \frac{d}{dt} \int \varphi_i^2 u^2 + \left(C_S^{-1} - A \delta \right) \left(\int (\varphi_i u)^4 \right)^{1/2} \right) & \leq \sum_i \left(C R^{-2} \int_{B_i } ( \sum_j \varphi_j^2) u^2 + B \int \varphi_i^2 u^2 \right) \\
& \leq \left( C N R^{-2} + B \right) \sum_i \int \varphi_i^2 u^2.
\end{split}
\end{equation*}
The interpolation inequality (7.10) of \cite{gt}, with $p=1, q=2, r=4,$ and $\mu = 2,$ yields
\begin{equation*}
\int (\varphi_i u)^2 \leq 2 \theta \left( \int \left(\varphi_i u \right)^4 \right)^{1/2} + 2 \theta^{-2} \left( \int \varphi_i u \right)^2.
\end{equation*}
Taking $\theta = R^2 \delta,$ we obtain
\begin{equation*}
\begin{split}
\sum_i \left( \frac{d}{dt} \int \varphi_i^2 u^2 \right. & \left. +\,\, 2 (C_S^{-1}  -  A\delta ) \left(\int (\varphi_i u)^4 \right)^{1/2} \right) \\
& \leq C\left(1 + B R^2 \right) \sum_i \left( \delta \left(\int (\varphi_i u)^4 \right)^{1/2} + \delta^{-2} R^{-6} \left( \int \varphi_i u \right)^2 \right).
\end{split}
\end{equation*}
Note that
$$\sum_i \left( \int \varphi_i u \right)^2 \leq \left( N \int u \right)^2.$$
Rearranging yields
$$ \sum_i \left(\frac{d}{dt} \int \varphi_i^2 u^2 + \epsilon \left( \int (\varphi_i u)^4 \right)^{1/2} \right) \leq  C \delta^{-2} R^{-6} \left( 1 + BR^2 \right) \left(\int u \right)^2$$
where
$$\epsilon = 2\left( C_S^{-1} - \delta \left( A+ C(1 + B R^2 ) \right) \right).$$

We now choose $R_0$ such that $B R_0^2 \leq 1,$ and let
$$\delta = \left(2C_S\left(A + 2C\right)\right)^{-1}.$$
These choices ensure that $\epsilon \geq \left(C_S\right)^{-1} > 0.$ Hence we may apply H\"older's inequality to the left-hand side and absorb the partition of unity. This yields
$$
\frac{d}{dt} \int u^2 + \dfrac{\epsilon }{c R^2} \int u^2 = \sum_i \left(\frac{d}{dt} \int \varphi_i^2 u^2 + \dfrac{\epsilon }{c R^2} \int \varphi_i^2 u^2 \right) \leq C R^{-6} \left( \int u \right)^2
$$
and
$$ \frac{d}{dt} \left(e^{\frac{\epsilon}{c R^2} t} \int u(t)^2 \right) \leq e^{\frac{\epsilon}{c R^2} t} \, C R^{-6} \left( \int u \right)^2.
$$
Integrating in time, we obtain
$$\int u(t)^2 \leq  e^{-\frac{\epsilon}{c R^2} t} \int u(0)^2 + C\epsilon^{-1} R^{-4} \left(1 - e^{-\frac{\epsilon}{c R^2} t} \right) \sup_{0\leq s \leq t} \left( \int u(s) \right)^2$$
which is equivalent to the desired bound.
\end{proof}


\begin{lemma}\label{moseriteration} \emph{(Parabolic Moser iteration)} Let $M$ be an $n$-dimensional Riemannian manifold and $B_1 \subset M$ a unit geodesic ball. For $0 < r \leq 1,$ denote the parabolic cylinder
$$P_r = B_{r} \times \left(-r^2, 0 \right).$$
Fix
$$0 < p \leq \infty \quad \quad \frac{n}{2} < q \leq \infty \quad \quad K > 0 \quad \quad 0 < \theta < 1$$
and assume that $u(x,t), f(x,t) \geq 0$ are functions on $P_1$ satisfying
$$\left( \partial_t + \Delta \right) u \leq f \cdot u$$
in the weak sense, with
$$\sup_{-1 < t < 0} \| f(\cdot, t) \|_{L^q \left(B_1 \right)} \leq K.$$ 
Then
$$\|u\|_{L^{\infty} \left(P_\theta \right)} \leq C_{\ref{moseriteration}} \|u\|_{L^p \left(P_1 \right) }.$$
The constant depends on $p, q, n, K, \theta, Vol(B_1),$ and the Sobolev constant of $B_1.$
\end{lemma}
\begin{proof} See Li \cite{li}, Lemma 19.1.
\end{proof}

\begin{prop}\label{moser} Let $R < \min (R_0, \sqrt{T} )$ 
and $u(t)$ be as in Proposition \ref{premoser}.
Choose $\tau > 0$ and let $\bar{\tau} = \tau / R^2.$ For any $p \geq 1$ and $x_0 \in M,$ there holds
\begin{equation*}
\| u (t )\|_{L^\infty(B_{R/2}(x_0))} \leq C_{\ref{moser}} \sup_{ t - \tau < s < t } \| u (s) \|_{L^p(B_R(x_0))} 
\end{equation*}
for $0 < \tau \leq t,$ where $C_{\ref{moser}} = C \left( R^{-4} \left(1 + \bar{\tau}^{-2} \right) \right)^{1/p}.$ 
\end{prop}
\begin{proof} We use a well-known argument due to Schoen and Uhlenbeck \cite{schoen}.
After rescaling
\begin{equation}\label{rescaling}
u(x,t) \to R^4 \, u(Rx, R^2 t)
\end{equation}
in geodesic coordinates about $x_0,$ we may assume that $u(x,t)$ is defined on $B_1 \times \LB 0 , 1 \RB.$ Then
$u$ satisfies
\begin{equation}\label{uevolution}
\begin{split}
\left( \partial_t + \Delta\right) u & \leq A u^{3/2} + R^2 B u \\
& \leq (A u^{1/2} + 1) \cdot u
\end{split}
\end{equation}
since we have chosen $R_0^2 B \leq 1,$ and
$$\int_{B_1} u(t) \leq \delta^2.$$
As in the proof of Proposition \ref{premoser}, for $R < R_0,$ the rescaled metric on $B_1$ (and any further rescaling) is close enough to Euclidean that we have uniform volume bounds and a uniform Sobolev constant $C_S,$ hence may apply the previous Lemma.

Let $P_r(x, t) = B_r(x) \times \LB t- r^2, t \RB,$ and abbreviate $P_r = P_r(0,1).$ Define
\begin{equation}\label{rescenergy}
e(r) = \left(1 - r\right)^4 \sup_{P_r} u
\end{equation}
and let $e_0, r_0$ be such that
\begin{equation}\label{e0}
e_0 = e(r_0) = \sup_{0 \leq r \leq 1} e(r).
\end{equation}
Choose $(x_1, t_1) \in P_r$ such that $u(x_1, t_1) = \sup_{P_r} u.$ Letting $\rho_0 = \left(1 - r_0 \right) / 2,$
we have
\begin{equation}\label{boundtorescale}
\left( \rho_0 \right)^4 \sup_{P_{\rho_0}(x_1, t_1)} u \leq 16 e_0.
\end{equation}

Assume first that $e_0 > 1.$ Letting
$$\rho_1 = \rho_0 \left(e_0\right)^{-1/4}$$ we may rescale
$$u_1(x, t) = \left( \rho_1 \right)^4 u \left( \rho_1 x + x_1 , \left( \rho_1\right)^2 \left( t - 1 \right) + t_1 \right)$$
to obtain a function $u_1$ of $(x, t) \in P_1.$
This again satisfies (\ref{uevolution}), and (\ref{boundtorescale}) implies
$$\sup_{P_1}u_1(x, t) \leq 16.$$
But then Lemma \ref{moseriteration}, applied to (\ref{uevolution}) with $p=1$ and $q = \infty,$ gives
\begin{equation}
\begin{split}
1 = u_1(0,1) & \leq C \int_0^1 \int_{B_1} u_1(x, t) \, dV dt \\
& \leq C \delta^2
\end{split}
\end{equation}
for a constant depending only on $A$ (which is universal). For $\delta$ sufficiently small, this is a contradiction.

Therefore $e_0 \leq 1.$ Directly from the definition (\ref{rescenergy}) and (\ref{e0}), we have for any $0 < r < 1$
$$\sup_{P_r}u = (1-r)^{-4} e(r) \leq (1-r)^{-4} e_0 \leq (1-r)^{-4}.$$
We may therefore apply Lemma \ref{moseriteration} to (\ref{uevolution}), to find
\begin{equation}\label{unitmoser}
\sup_{P_{1/2}} u \leq C \left( \int_0^1 \int_{B_{3/4}} u^p \, d V dt \right)^{1/p}.
\end{equation}
If $\bar{\tau} < 1,$ we rescale by an additional factor $\sqrt{\bar{\tau}},$ and again obtain (\ref{unitmoser}). Overall, undoing the rescaling, we have
\begin{equation}\label{detailedbound}
\|u (t) \|_{L^\infty \left(B_{R/2}(x_0) \right)} \leq C \left( R^{-6} \left( 1 +\bar{\tau}^{-3} \right) \int_{t-\tau}^{t} \int_{B_R(x_0)} u(s)^p\, dV ds \right)^{1/p}
\end{equation}
for $t \geq \tau.$ The desired estimate 
follows directly from (\ref{detailedbound}) with $\bar{\tau} \leq 1.$
\end{proof}

\subsection{Criterion for long-time existence}

\begin{lemma}\label{cutoff} (C.f. \cite{donkron}, 7.2.10). There is a universal constant $C$ and for any $N \geq 2, R > 0,$ a smooth radial 
function $\beta = \beta_{N,R}$ on $\R^4,$ with
$$0 \leq \beta(x)\leq 1$$
\begin{equation*}
\beta(x)=
\begin{cases}
1 & |x| \leq R/N \\
0 & |x| \geq R
\end{cases}
\end{equation*}
and
$$\|\nabla \beta \|_{L^4} + \|\nabla^{(2)} \beta\|_{L^2} < \frac{C}{\sqrt{\log N}}.$$


Assuming $R < R_0,$ the same holds for $\beta ( x - x_0 )$ on any geodesic ball $B_R(x_0) \subset M.$
\end{lemma}
\begin{proof} We take
$$\beta(x) = \tilde{\varphi}\left(\dfrac{\log \frac{N}{R} \,|x|}{\log N} \right)$$
where 
\begin{equation*}
\tilde{\varphi}(s)=
\begin{cases}
1 & s \leq 0 \\
0 & s \geq 1
\end{cases}
\end{equation*}
is a standard cutoff function, with respect to the cylindrical coordinate $s.$
\end{proof}

\begin{thm}\label{selfdual} Assume that $A(t)$ is a smooth solution of the Yang-Mills flow on $M$ for $0 \leq t < T,$ and write $F(t) = F_{A(t)}.$ Let $B_{R/N} \subset B_R$ be concentric geodesic balls in $M,$ with $R < R_0$ and $N \geq 2.$ There holds
\begin{equation}\label{selfdualbound}
\| F(t) \|^2_{L^2(B_{R/N})} \leq \|F(0)\|^2_{L^2(B_{R})} + \int_0^{t} \frac{\|F^+(s)\|_{L^\infty(B_R)} }{\sqrt{\log(N)}} \left(C+\|F^-(s)\|^2_{L^2(B_R)}\right) ds.
\end{equation}
Therefore, if 
$$\|F^+(t)\|_{L^\infty(M)} \in L^1 \left( \LB 0, T \right) \right)$$
or, in particular, if $F^+$ does not concentrate in $L^2$ as $t\to T,$ then $\lim_{t \to T} A(t)$ exists in $C^\infty(M)$ and the flow extends smoothly.
\end{thm}
\begin{proof}
Recall the evolution of the curvature tensor
$$\frac{\partial}{\partial t} F + D D^* F = 0.$$ 
Taking an inner-product with $\varphi^2 F$ and integrating by parts, we obtain
\begin{equation*}
\frac{1}{2} \frac{d}{dt} \int \varphi^2  |F|^2 \, dV + \int \LA D^*(\varphi^2 F) , D^*F \RA \, dV = 0.
\end{equation*}
Note from (\ref{adjoint}) that 
\begin{equation*}
\begin{split}
D^*(\varphi^2 F)_j & = - \nabla^k (\varphi^2 F_{kj} ) \\ 
& = -2 \varphi \nabla^k \varphi F_{kj} + \varphi^2 D^*F_j.
\end{split}
\end{equation*}
Abbreviating $\|\cdot \| = \|\cdot \|_{L^2(M)}$ as before, we therefore have
\begin{equation*}
\frac{1}{2} \frac{d}{dt} \|\varphi F\|^2 + \|\varphi D^* F\|^2 = 2 \int \LA \varphi \nabla^k \varphi F_{kj} , D^*F^j \RA \, dV.
\end{equation*}

On the right-hand side we substitute $D^* F = 2D^* F^+,$ 
and integrate by parts again, to obtain
$$\frac{1}{2} \frac{d}{dt} \|\varphi F\|^2 + \|\varphi D^* F\|^2 = 4 \int \Big\LA \left( \nabla_i \varphi \nabla^k\varphi + \varphi \nabla_i \nabla^k \varphi \right) F_{kj} + \varphi \nabla^k\varphi \nabla_i F_{kj} \, , (F^+)^{ij} \Big\RA \,\, dV. $$
In the inner product with the self-dual two-form $F^+,$ we may replace the term $\varphi \nabla^k \varphi \nabla_i F_{k j}$ via the identity
\begin{equation*}
\begin{split}
\left( \nabla^k \varphi \left(\nabla_i F_{k j} - \nabla_j F_{k i} \right) \right)^+ & = \left( \nabla^k \varphi \left( \left( - \nabla_j F_{i k} - \nabla_k F_{ji} \right) - \nabla_j F_{k i} \right) \right)^+ \\
& = \left( \nabla^k \varphi \nabla_k F_{ij} \right)^+ \\
& = \nabla^k \varphi \nabla_k F^+_{ij}.
\end{split}
\end{equation*}
We then write
$$\big\LA \nabla_k F^+_{ij}, (F^+)^{ij} \big\RA = \frac{1}{2}\nabla_k|F^+|^2$$
and integrate by parts once more, to obtain
\begin{equation*}
\begin{split}{}
\frac{1}{2} \frac{d}{dt} \|\varphi F\|^2 + \|\varphi D^* F\|^2 & = 4 \int \left( \nabla_i\varphi \nabla_k \varphi + \varphi \nabla_i \nabla_k \varphi \right) \left( \Big\LA F^k{}_j, (F^+)^{ij} \Big\RA - g^{ik}\frac{|F^+|^2}{4} \right) \,\, dV \\[2mm]
& = 4 \int \left( \nabla_i\varphi \nabla_k \varphi + \varphi \nabla_i \nabla_k \varphi \right) \big\LA \left( F^- \right)^k{}_j, (F^+)^{ij} \big\RA \,\, dV
\end{split}
\end{equation*}
where the second identity follows from a calculation as in Section \ref{weitzsection}. Removing an $L^\infty$ norm, and applying Young's inequality, yields
$$\frac{d}{dt} \|\varphi F\|^2 \leq  8 \, \|F^+\|_{L^\infty(B_r)} \left( \epsilon^{-1} \|F^-\|^2_{L^2} + \epsilon \left( \|\nabla \varphi\|_{L^4}^4 + \|\varphi \nabla^2 \varphi \|^2_{L^2} \right) \right).$$
Choose $\epsilon= 8 \sqrt{\log(N)} $ and $\varphi=\beta_{N,r},$ from Lemma \ref{cutoff}, to obtain the desired estimate (\ref{selfdualbound}).

By Theorem \ref{struwe}, to prove the second claim, it suffices to show that the full curvature does not concentrate in $L^2$ at time $T.$ Note that $\|F^-(t)\|_{L^2(B_R)}^2$ is bounded above by the total anti-self-dual energy, which by (\ref{energyformula}) is decreasing. Therefore if the full curvature on $B_R$ is initially less than
$\epsilon_0/2,$ then for $N$ sufficiently large, (\ref{selfdualbound}) implies that the full curvature on $B_{R/N}$ will remain less than $\epsilon_0$ at time $T.$


Lastly, by Proposition \ref{moser}, non-concentration of $F^+$ implies a uniform $L^\infty$ bound, and hence the required $L^1(L^\infty)$ bound on the self-dual energy at finite time.
\end{proof}

\begin{cor}\label{tube} Let $\delta$ be the universal constant of Proposition \ref{moser}. If an initial connection has self-dual curvature $\|F^+\|_{L^2(M)} < \delta,$ then the Yang-Mills flow exists for all time, with curvature blowing up at most exponentially as $t \to \infty.$ 
\end{cor}

\begin{cor}\label{plusminus} If the maximal existence time, $T,$ is finite, then both $F^+$ and $F^-$ must concentrate at some point 
$x_0 \in M$ as $t \to T.$
\end{cor}


\begin{rmk}\label{simul}
The proof of Theorem \ref{selfdual} gives a more refined long-time existence criterion, which will be the basis for a forthcoming paper. 
Let
$$S_{ij} = \LA F_i{}^k , F_{jk} \RA - \frac{1}{4}|F|^2 g_{ij} = 2 \LA F^+_i{}^k , F^-_{jk} \RA $$
be the stress-energy tensor for Yang-Mills, and define
$$N(x, t) = x^i x^j S_{ij}.$$
Here $x^i$ are geodesic coordinates centered at $x_0 \in M.$ 
\begin{thm} If, for some $r_0 > 0,$ there holds
$$\lim_{t \to T} \sup_{0 < r < r_0} \left| \dashint_{S^3_r(x_0)} N(x, t) \, d S_x \right| < \infty$$
then no singularity occurs at $(x_0, T).$
\end{thm}
\end{rmk}

\section{Convergence at infinite time}\label{convergencesection}

We now turn to the question of convergence of the Yang-Mills flow as $t\to \infty,$ making the assumption
\begin{equation}\label{lowselfdual}
\|F_{A}^+\|_{L^2(M)} < \delta.
\end{equation}
By (\ref{energyformula}), this condition is preserved by the flow, which, according to Corollary \ref{tube}, exists for all time. Moreover, if $E$ has structure group $SU(2),$ then
(\ref{lowselfdual}) should represent the generic end-behavior \cite{bl}. 

We are particularly concerned with the interplay between weak Uhlenbeck limits, which are taken modulo gauge away from bubble points, 
and smooth convergence of the flow. Lemma \ref{antibubble} provides 
a rapid proof in the present context, as in that of Kahler surfaces, that the former always exist and are Yang-Mills. Because of the $(T-\tau)$ factor in the estimates of Proposition \ref{omega}, however, 
fast decay 
of the energy is needed in order to conclude that the flow converges smoothly.\footnote{This subtlety also obstructs an obvious proof of long-time existence using only Uhlenbeck compactness and $\epsilon$-regularity
to produce a bubble tree and control it.} 
A cohomological assumption on the Uhlenbeck limit will supply the required estimate (\ref{poincare}) along the flow. 

For the remainder of the paper, $\Omega \subset M$ will denote an open set, and we let
$$\Omega_r = \{ x \in \Omega \mid d(x,\Omega^c) >r \} \subset \subset \Omega.$$
For points $x_1, \ldots, x_m \in M,$ write
$$M_0 = M \setminus \{ x_1, \ldots, x_m \}$$
and
$$M_r = \left(M_0 \right)_r = M \setminus \bar{B}_r(x_1)\cup \cdots \cup \bar{B}_r(x_m).$$
As before, we will abbreviate $\|\cdot\| = \|\cdot\|_{L^2(M)}.$ 

\subsection{Basic estimates}\label{estimatesection}

\begin{lemma}\label{epsilonreg} \emph{($\epsilon$-regularity)} There exists $\epsilon_0 > 0$ as follows. For $R < R_0,$ if
\begin{equation}\label{epsilonregassumption} \|F(t)\|^2_{L^2(B_R)} < \epsilon_0
\end{equation}
for all times $t$ with $-R^2 \leq t < 0,$ 
then there holds
\begin{equation*}\label{epsilonregestimate}
\|\nabla_A^{(k)} F(t)\|_{L^\infty \left(B_{R_k} \right)} < \frac{C_{\ref{epsilonreg}}}{R^{2+k}}
\end{equation*}
for $k \geq 0$ and $- R_k^2 \leq t < 0,$ where $R_k = R/2^{k+1}.$ The constant depends only on $k.$
\end{lemma}
\begin{proof} The $k=0$ bound follows from Proposition \ref{moser}, with $p=1$ and $\bar{\tau} = 1.$\footnote{Via the monotonicity formula \cite{hamilton}, it suffices to assume (\ref{epsilonregassumption}) only at $t = -R^2$ (see \cite{chenshen}, \cite{hongtian}).}
For $k \geq 1,$ this is the result of the Bernstein-Hamilton-type derivative estimates of \cite{weinkove}, Theorem 2.2.
\end{proof}

\begin{prop}\label{dstarf} Assume $\|F(t)\|_{L^\infty(B_{R}(x_0))} < K$ for $0 \leq t < T.$ Then for $\tau > 0, R< R_0,$ we have
\begin{equation*}
\begin{split}
\|\nabla^{(k)} D^*F(t)\|_{L^\infty (B_{R_k})}^2 & \leq C_{\ref{dstarf}} \|D^* F\|^2_{L^2( B_R \times \LB t-\tau, t \RB )}
\\[2mm]
\|\nabla^{(k)} F(t)\|_{L^\infty (B_{R_k})}^2 & \leq C_{\ref{dstarf}} \left( \|D^* F\|^2_{L^2( B_R \times \LB t-\tau, t \RB )} + \|F\|_{L^2( B_R \times \LB t-\tau, t \RB )}^2 \right)
\end{split}
\end{equation*}
for $k \geq 0$ and $k \tau \leq t < T.$
The constants depend on $K, k, R,$ and $\tau.$
\end{prop}
\begin{proof} One computes the evolution
\begin{equation}\label{dstarfevolution}
\begin{split}
\frac{\partial}{\partial t} \left( D^*F \right)_i & = -\frac{\partial}{\partial t} \nabla^k F_{ki} \\
& = - \LB (-D^*F)^k, F_{ki} \RB - (D^*DD^* F)_i \\
& = \LB F_i{}^k , \left( D^*F\right)_k \RB - \left( \Delta_A D^*F \right)_i \\
& = \nabla^k \nabla_k D^*F_i + 2 \LB F_i{}^k , D^*F_k \RB + Rm \# D^*F.
\end{split}
\end{equation} 
In the third line, we used the identity (\ref{firstbianchi}) to obtain the Hodge Laplacian. Multiplying (\ref{dstarfevolution}) by $D^*F$ gives
$$\left(\partial_t + \Delta\right) |D^*F|^2 \leq C\left( 1 + K\right) |D^*F|^2.$$
The first estimate, with $k=0,$ follows from Lemma \ref{moseriteration} applied to (\ref{dstarfevolution}). 
Applying a cutoff for $B_{3R_1/2} \subset B_{R_0}$ and using Young's inequality, as well as the $k=0$ estimate, one also obtains
\begin{equation}\label{step1}
\begin{split}
\int_{t-\tau/2}^{t} \|\nabla D^*F(s)\|^2_{L^2\left( B_{3R_1/2} \right)} ds \leq C \|D^*F\|^2_{L^2(B_R \times \LB t-\tau, t \RB)}.
\end{split}
\end{equation}

Next, applying $\nabla$ to (\ref{dstarfevolution}), one obtains an evolution equation
\begin{equation}\label{nastyevolution}
\left( \partial_t + \nabla^* \nabla \right) \nabla D^*F  = F \# \nabla D^*F + Rm \# \nabla D^*F + \nabla F \# D^*F + \nabla Rm \# D^*F.
\end{equation}
Note from Lemma \ref{epsilonreg}
that all derivatives of $F$ are bounded in terms of $K.$ Multiplying (\ref{nastyevolution}) by $\nabla D^*F$ and again applying Lemma \ref{moseriteration}, we bound $\|\nabla D^*F(t)\|_{L^\infty\left( B_{R_1}\right)},$ for $t \geq \tau,$ by the LHS of (\ref{step1}), which concludes the $k=1$ case. The higher derivative estimates proceed by induction, using formulae similar to (\ref{nastyevolution}).

The argument for the second estimate is identical, beginning with the inequality
$$\|\nabla F\|_{L^2(B_{R_0})}^2 \leq C \left(\|D^*F\|_{L^2(B_{R})}^2 + \|F\|_{L^2(B_{R})}^2 \right)$$
which follows from the Weitzenbock formula (\ref{weitz}).
\end{proof}


\begin{prop}\label{omega} Let $0 < \tau_0 \leq \tau < T,$ and assume $\|F(t)\|_{L^\infty (\Omega)} < K$ for $ \tau - \tau_0 \leq t < T.$ Then we have the $L^\infty$ bound
$$\|A(T) - A(\tau)\|^2_{L^\infty(\Omega_r)} \leq C_{\ref{omega}} \left( \|F(\tau - \tau_0)\|^2 - \|F(T)\|^2 \right) \left( T-\tau \right)$$
as well as the Sobolev bounds
\begin{equation}
\begin{split}
\|A(T) - A(\tau)\|^2_{H^k(\Omega_r)} & \leq C_{\ref{omega}} \left( \|F(\tau - \tau_0)\|^2 - \|F(T)\|^2 \right) \left( T-\tau \right) \\
& \quad \quad \quad \quad \quad \quad \quad \quad \cdot \left(1 + \sup_{\tau \leq t < T} \|A(t) \|^{2k}_{H^{k-1}\left( \Omega  \right) } \right)
\end{split}
\end{equation}
for $k \geq 1.$
The constants depend on $K, k, r, \tau_0, \Omega \subset M,$ and $\nabla_{ref}$ (for $k \geq 1$).
\end{prop}
\begin{proof} For the first bound, we calculate
\begin{equation}\label{omega1}
\begin{split}
\|A(T) - A(\tau)\|_{L^\infty \left( \Omega_r \right)} & \leq \int_\tau^T \|D^*F(t)\|_{L^\infty(\Omega_r)} \, dt \\
& \leq C_{\ref{dstarf}} \int_\tau^T \|D^*F\|_{L^2 \left(\Omega \times \LB t-\tau_0, t \RB \right)} \, dt \\
& \leq C (T-\tau)^{1/2} \left( \int_\tau^T \|D^*F\|^2_{L^2\left( \Omega \times \LB t-\tau_0, t \RB \right)} \, dt \right)^{1/2} \\
& \leq C (T - \tau)^{1/2} \left( \int_\tau^T \int_{t - \tau_0}^t \|D^*F(s) \|^2 \, ds dt \right)^{1/2}.
\end{split}
\end{equation}
The domain of integration
$$t - \tau_0 \leq s \leq t \quad \quad \tau \leq t \leq T$$
may be relaxed to
$$\tau - \tau_0 \leq s \leq T \quad \quad s \leq t \leq s + \tau_0.$$ Then (\ref{omega1}) becomes
\begin{equation*}
\begin{split}
\|A(T) - A(\tau)\|_{L^\infty \left( \Omega_r \right)} & \leq C (T - \tau)^{1/2} \tau_0^{1/2} \left( \int_{\tau - \tau_0}^T \|D^*F(s)\|^2 \, ds \right)^{1/2} \\
& \leq C (T - \tau)^{1/2} \left( \|F(\tau - \tau_0)\|^2 - \|F(T)\|^2 \right)^{1/2}
\end{split}
\end{equation*}
as desired. For $k=1,$ write
\begin{equation*}\label{derivcomparison}
\begin{split}
\partial_t \nabla_{ref} A & = - \nabla_{ref} D^*F \\
& = - \nabla_A D^*F + A \# D^*F.
\end{split}
\end{equation*}
We then have
$$\| \nabla_{ref} \left( A(T) - A(\tau) \right) \|_{L^2} \leq C \left( 1 + \sup \| A \|_{L^2} \right) \int_\tau^T \left( \|D^*F(t)\|_{L^\infty} + \| \nabla_A D^*F(t) \|_{L^\infty} \right) dt $$
and may apply Proposition \ref{dstarf} as above.
For $k=2,$ write
\begin{equation*}
\begin{split}
\partial_t \nabla_{ref}^{(2)} A & = - \nabla_{ref}^{(2)} D^*F \\
& = \nabla_A^{(2)} D^*F + A \# \nabla_A D^*F + \nabla_{ref}A \# D^*F + A \# A \# D^*F
\end{split}
\end{equation*}
and note that $\|A\|_{L^2} + \|\nabla_{ref}A \| _{L^2} + \|A \# A \|_{L^2} \leq C \left( 1 + \|A\|_{H^1}^2 \right).$
The higher derivative bounds proceed similarly.
\end{proof}

\subsection{Uhlenbeck limits}\label{uhllimsection}
For a sequence $t_j \to \infty,$ we say that $(E_\infty, A_\infty)$ is an \emph{Uhlenbeck limit} along the flow if there exists a subsequence of times $t_{j_k}$ and smooth bundle isometries $u_k : E \to E_\infty,$ defined on an exhaustion of open sets
$$U_1 \subset \cdots \subset U_k \subset \cdots \subset M_0 = M \setminus \{ x_1, \ldots, x_m \}$$
such that
$$u_k (A(t_{j_k})) \to A_\infty$$ 
smoothly on any $\Omega \subset \subset M_0.$

\begin{lemma}\label{antibubble} Assume $\|F^+(t)\|_{L^\infty(\Omega)} < K^+$ for $0 \leq t \leq \tau.$ Let $\epsilon_0$ be as in Lemma \ref{epsilonreg}, and assume that for some $0 < r_0 < R_0$ there holds
\begin{equation}\label{epsilonover3}
\|F(\tau)\|^2_{L^2(B_{r_0}(x_0))} < \epsilon_0 / 3
\end{equation}
for all $x_0 \in \Omega_{r_0},$ with $0 < r_0^2 < \tau.$ If
\begin{equation}\label{energynondecrease}
\|F(0)\|_{L^2(M)}^2 - \|F(\tau)\|_{L^2(M)}^2 \leq \epsilon_0 / 3
\end{equation}
then we have
$$\|\nabla_A^{(k)} F(\tau)\|_{L^\infty(\Omega_{r_0})} < \frac{C_{\ref{antibubble}}}{r_0^{2 + k}} $$
for $k \geq 0.$ The constant depends on $K^+, \|F(0)\|,$ and $k.$
\end{lemma}
\begin{proof} 
Let $x_0 \in \Omega_{r_0},$ and $\varphi$ be the cutoff of Lemma \ref{cutoff} for $B_{r_0/N}(x_0) \subset B_{r_0}(x_0).$ We apply the proof of Theorem \ref{selfdual} using $\overline{\varphi} = 1-\varphi.$ This gives
\begin{equation}\label{backwardselfdual}
\|F(\tau)\|^2_{L^2(M \setminus B_{r_0})} - \|F(t)\|^2_{L^2(M \setminus B_{r_0/N})} < \epsilon_0 / 3
\end{equation}
for $N$ large enough based on $\|F\|^2$ and $K^+,$ but independent of $x_0$ and $r_0.$ Adding (\ref{energynondecrease}), with $t$ in place of zero, and (\ref{backwardselfdual}), we obtain
$$\|F(t)\|^2_{L^2(B_{r_0/N})} - \|F(\tau)\|^2_{L^2(B_{r_0})} < 2\epsilon_0/3.$$
By (\ref{epsilonover3}), we have
$$\|F(t)\|^2_{L^2(B_{r_0/N})} < \epsilon_0$$
for $0 \leq t \leq \tau.$ The desired bounds follow from Lemma \ref{epsilonreg}.
\end{proof}


\begin{thm}\label{uhllimexist} Assume $\|F^+(t)\|_{L^\infty(M)} < K^+$ for $t$ sufficiently large. 
For any sequence $t_j \to \infty,$ there exists an Uhlenbeck limit, and any such limit is Yang-Mills. 
\end{thm}
\begin{proof} This is a direct adaptation of the arguments contained in Donaldson and Kronheimer \cite{donkron}, \S 4 and \S 6.2.4, as follows. The existence of weak $H^1$ limits on a countable family of balls in $M_0$ is the result of compactness theory for connections with bounded $L^2$ curvature in Coulomb gauge (\cite{sedlacek}, \cite{uhlenbecklp}). By Lemma \ref{antibubble}, we in fact have $L^\infty$ bounds on the curvature of $A(t_{j_k})$ and all its derivatives on each ball, for $k$ large enough. By \cite{donkron}, Lemma 2.3.11, the weak limits can be taken to be smooth limits over each ball; and by \cite{donkron}, Corollary 4.4.8, the gauge transformations can be patched together over the open sets $U_i.$\footnote{Note that Theorem 1.3(ii) of Schlatter \cite{schlatterlongtime}, which finds a Yang-Mills connection as a weak limit along a specially chosen sequence of times, does not include any patching, as this may not be possible with $H^2$ gauge transformations.}

The fact that the limiting connection is Yang-Mills away from the bubbling points, and therefore extends to a smooth Yang-Mills connection on $E_\infty,$ follows directly 
from the energy inequality, Proposition \ref{dstarf}, and \cite{uhlenbeckremov}.
\end{proof}

\subsection{Sobolev and Poincar\'e inequality for self-dual forms}\label{sobpoinc}
The following brief discussion amounts to the key analytic observation of Taubes \cite{taubes} (see also \cite{donkron}, \S 7).

Assuming $\|F_A^+\| < \delta,$ H\"older's inequality applied to the Weitzenbock formula (\ref{sdweitz}) implies, for $\omega \in \Omega^{2+}(\gothg_E),$ the Sobolev inequality
\begin{equation}\label{sobolev}
\begin{split}
\|\omega\|_{L^4(M)}^2 + \|\nabla_A \omega\|^2 & \leq C_M \left( \|D_A \omega\|^2 + \|D_A^* \omega \|^2 + \|\omega\|^2 \right) \\
& \leq C_M \left( \|D_A^* \omega \|^2 + \|\omega\|^2 \right).
\end{split}
\end{equation}
In the second line, we used the pointwise identity
$$|D_A \omega| = |-*D_A*\omega| = |D_A^*\omega|.$$

In case $A$ is an instanton, 
recall the basic complex (\cite{donkron}, \S 4.2.5) 
$$
0 \xrightarrow{\hspace*{.4cm}} \gothg_E \stackrel{D_A}{\xrightarrow{\hspace*{.7cm}}} \Omega^1(\gothg_E) \stackrel{D^+_A}{\xrightarrow{\hspace*{.7cm}}} \Omega^{2+}(\gothg_E) \xrightarrow{\hspace*{.4cm}} 0.
$$
The assumption of vanishing second cohomology group, which we will write $H^{2+}_A = 0,$ is equivalent to the statement
$$D_A^* \omega = 0 \implies \omega = 0.$$
The usual compactness argument then 
gives an inequality
\begin{equation*}
\|\omega\|^2 \leq C_A \|D_A^* \omega \|^2
\end{equation*} 
for $\omega \in \Omega^{2+}(\gothg_E).$ Hence this term can be dropped from the RHS of (\ref{sobolev}), yielding
\begin{equation*}
\|\omega\|_{L^4}^2 + \|\omega\|^2 + \|\nabla_A \omega\|^2 \leq C_A \|D_A^* \omega \|^2.
\end{equation*}

We require only the \emph{Poincar\'e inequality}
\begin{equation}\label{poincare}
\|\omega\|_{L^4}^2 + \|\omega\|^2  \leq C_{A} \|D_{A}^* \omega \|^2
\end{equation}
for $\omega \in \Omega^{2+}(\gothg_E),$ where we always take $C_A \geq C_M.$ Here $A$ need not be an instanton. This inequality has the following basic stability property in dimension four.


\begin{lemma}\label{excision} Fix $x_1, \ldots, x_m \in M,$ and let $A_0$ be a connection on a bundle $E_0 \to M$ which satisfies (\ref{poincare})
with constant $C_0 = C_{A_0}.$ Assume that $A$ is a connection on $E$ with $\|F_A^+\| < \delta,$ for which there exists a smooth bundle isometry $u : E_0 \to E$ defined over $M_r = M \setminus \bar{B}_r(x_1) \cup \cdots \cup \bar{B}_r(x_m),$ 
such that
$$\|u(A)-A_0\|_{L^4(M_r)} \leq \epsilon.$$
If $\epsilon$ and $r$ are sufficiently small, depending only on $C_0, R_0,$ and $m$,
then $A$ satisfies (\ref{poincare}) with constant $8 C_0.$
\end{lemma}
\begin{proof} Assume first that $\mbox{Supp}(\omega) \subset M_r.$ Write $\tilde{A} = u(A), \tilde{\omega} = u(\omega),$ $a = A_0 - \tilde{A}.$ We then have
$$
\|D_{A}^* \omega \|^2 = \|D_{\tilde{A}}^* \tilde{\omega} \|^2 = \|D_{A_0}^* \tilde{\omega} + a \# \tilde{\omega}  \|^2 \\
$$
and
$$\|D_{A_0}^* \tilde{\omega} \|^2 \leq 2\left( \|D_{A}^* \omega \|^2 + \| a \|_{L^4}^2 \|\omega\|_{L^4}^2 \right).$$
On the other hand, if $\mbox{Supp}(\omega) \subset B_r(x_1) \cup \cdots \cup B_r(x_m),$ then
$$\|\omega \|^2 \leq c m r^2 \|\omega \|^2_{L^4}.$$

Choose $\epsilon, r, N$ such that
$$4\epsilon^2 + c m r^2 + C/\log(N) < (8 C_0)^{-1}.$$
Let $\varphi = \sum \beta_{N,r}(x-x_i)$ be a sum of the logarithmic cutoffs of Lemma \ref{cutoff}, and $\overline{\varphi} = 1-\varphi.$
Combining the above observations, we have
\begin{equation*}
\begin{split}
\|\omega\|_{L^4}^2 + \|\omega\|^2 & \leq 2 \left( \|\varphi \omega\|_{L^4}^2 + \|\varphi \omega\|^2 + \|\overline{\varphi} \omega\|_{L^4}^2 + \|\overline{\varphi} \omega\|^2 \right) \\
& \leq 2 C_M \left( \|D_A^* (\varphi \omega ) \|^2 + \|\varphi \omega\|^2 \right) + 2 C_0 \| D^*_{A_0} (\overline{\varphi} \tilde{\omega}) \|^2 \\
& \leq 4 C_0 \left( \|\varphi D_A^* \omega\|^2 + \|\overline{\varphi} D_A^* \omega\|^2 + 2\|D\varphi \# \omega\|^2 + \left( 4 \|a\|_{L^4}^2 + cmr^2 \right) \|\omega\|_{L^4}^2 \right) \\
& \leq 4 C_0 \left( \|D_A^* \omega\|^2 + \left(2\|D\varphi\|_{L^4}^2 + 4 \epsilon^2 + cmr^2\right) \|\omega\|_{L^4}^2 \right).
\end{split}
\end{equation*}
Rearranging yields the desired estimate, where we replace $r/N$ by $r.$
\end{proof}

\subsection{Convergence}\label{convergencesubsection}

\begin{thm}\label{convergence} Fix $x_1, \ldots, x_m \in M,$ and let $A_0, E_0, \epsilon,$ and $r$ be as in Lemma \ref{excision}. Choose $\tau_0 > 0$ and $0 < r_0 < \min (r/3, R_0, \sqrt{\tau_0}).$ There exists $\bar{\delta}_1 > 0$ as follows.

Assume that $A(t)$ solves the Yang-Mills flow with $\| F^+_{A(0)} \| = \| F^+(0) \| < \delta,$ and for some $\tau \geq \tau_0,$ there hold
\begin{equation}\label{convergencelowselfdual}
\| F^+(\tau - \tau_0) \| \leq \delta_1 < \bar{\delta}_1
\end{equation}
\begin{equation}\label{convergencenonconc}
\| F(\tau) \|_{L^2(B_{r_0}(x))} < \epsilon_0 / 3 \quad \forall x \in M_{2r/3}
\end{equation}
and, for a bundle isometry $u: E \to E_0$ over $M_{r/3}$
\begin{equation}\label{closeness}
\| u(A(\tau)) - A_0 \|_{L^4(M_{r/3})} \leq \epsilon_1 < \epsilon / 2.
\end{equation}
Then the flow converges smoothly 
on $M$ as $t\to \infty,$ with limit an instanton $A_\infty$ on $E.$ We have also the bounds
$$\| u(A_\infty) - A_0 \|_{L^4(M_r)} \leq C_{\ref{convergence}} \, \delta_1 + \epsilon_1$$
$$\| A(t) - A_\infty \|_{L^\infty(M_r)} \leq C_{\ref{convergence}} \, \delta_1 \, e^{-(t-\tau)/C_1}$$
for $t \geq \tau,$ and
\begin{equation*}\label{convhigherderivs}
\begin{split}
\quad \quad \quad \| A(t) - A_\infty \|_{H^k(M_r)} & \leq C_{\ref{convergence}} \, \delta_1 \, e^{-(t-\tau)/C_1}  \left( 1 + \| A(\tau) \|^{k!}_{H^{k-1}(M_{r/2})} \right)
\end{split}
\end{equation*}
for $k \geq 1.$

The constants $C_{\ref{convergence}}$ and $\bar{\delta}_1$ depend on $C_{A_0}, r_0, \tau_0, m, k, \kappa(E),$ $\nabla_{ref}$ (for $k \geq 1$), and the geometry of $M.$ The constant $C_1$ is a universal multiple of $C_{A_0}.$ 
\end{thm}
\begin{proof} Recall, from Corollary \ref{tube}, that the solution $A(t)$ exists for all time and is smooth. 
By Lemma \ref{excision}, the inequality 
\begin{equation}\label{convpoincare}
\|F^+(t)\|^2 \leq 8C_{A_0} \|D^* F^+(t) \|^2
\end{equation} 
holds for $A(t)$ on some maximal time 
interval $ \tau \leq t < T.$ We will argue that if $\delta_1$ is sufficiently small, then $T = \infty$ and the flow converges. Assume, for contradiction, that $T < \infty.$

Applied to the global energy inequality, 
 (\ref{convpoincare}) yields
$$\frac{d}{dt}\|F^+\|^2 + C_1^{-1} \|F^+\|^2 \leq \frac{d}{dt} \|F^+\|^2 + 4 \|D^*F^+\|^2 = 0. $$ 
In view of (\ref{convergencelowselfdual}), this implies 
the exponential bound 
\begin{equation}\label{decay}
\frac{1}{2}\left(\|F(t)\|^2 - \|F(T)\|^2 \right) \leq \|F^+(t)\|^2 \leq \delta_1^2 \, e^{-(t-\tau)/C_1}
\end{equation}
for $\tau - \tau_0 \leq t \leq T.$ Proposition \ref{moser}, with $p=1,$ then implies 
the global $L^\infty$ bound
\begin{equation}
\|F^+(t)\|^2_{L^\infty(M)} \leq C_{\ref{moser}} \, \delta_1^2 \, e^{-(t-\tau)/C_1} = :  K^+(t)^2.
\end{equation}
Therefore, if $\delta_1$ is sufficiently small, we have
\begin{equation}\label{kappa}
\left( C + \|F(0)\|^2 \right) \int_\tau^T K^+(t) dt < \epsilon_0/3.
\end{equation}
By Theorem \ref{selfdual} and (\ref{convergencenonconc}), the full curvature cannot concentrate on $M_{2r/3}$ before time $T,$ and we have a uniform bound
\begin{equation}\label{curvbound}
\|F(t)\|_{L^\infty(M_{2r/3})} < K
\end{equation}
for $\tau+r_0^2 < t < T.$

In order to apply Proposition \ref{omega}, we need this curvature bound on $M_{2r/3}$ also from time $\tau - r_0^2/2.$ Note that
\begin{equation}\label{preenergy}
\delta_1^2 \geq \|F^+(\tau - r_0^2)\|^2 \geq \frac{1}{2}\left(\|F(\tau-r_0^2)\|^2 - \|F(T)\|^2 \right).
\end{equation}
Hence, provided that we choose $\delta_1^2 < \epsilon_0 / 6$ and enlarge $K$ appropriately, Lemma \ref{antibubble} gives a uniform bound of the form (\ref{curvbound}) for $\tau-r_0^2/2 < t \leq \tau + r_0^2.$ 

We may now apply Proposition \ref{omega} and (\ref{decay}) 
at each time $\tau + i,$ to conclude
\begin{equation}\label{preconvlinfty}
\|A(\tau + i + 1) - A(\tau + i)\|^2_{L^\infty(M_r)} \leq C_{\ref{omega}} \left( K^+(\tau + i) \right)^2.
\end{equation}
By the triangle inequality and geometric series, we have
\begin{equation}\label{convlinfty}
\begin{split}
\|A(T) - A(\tau)\|_{L^\infty(M_r)} & \leq C \sum K^+(\tau + i) \\
& \leq C K^+(\tau) = C \delta_1.
\end{split}
\end{equation} 
If $\delta_1$ is small enough that $C \delta_1 < \epsilon/2,$ we conclude
\begin{equation*}
\begin{split}
\|u(A(T)) - A_0 \|_{L^4(M_r)} & \leq \|u(A(T)) - u(A(\tau))\|_{L^4(M_r)} + \|u(A(\tau)) - A_0 \|_{L^4(M_{2r/3})} \\
& \leq C  \delta_1 + \epsilon_1 < \epsilon.
\end{split}
\end{equation*}
Hence, by Lemma \ref{excision}, the inequality (\ref{convpoincare}) continues beyond $t=T,$ which is a contradiction.

Therefore $T = \infty,$ and the above estimates continue as $t\to \infty.$ Letting $\tau$ increase, (\ref{convlinfty}) implies exponential convergence  in $L^\infty(M_r)$ to a limit $A_\infty,$ as desired. For the Sobolev bounds, note that (\ref{preconvlinfty}) and (\ref{convlinfty}), with $H^k$ in place of $L^\infty,$ give
\begin{equation}\label{convbootstrap}
\| A(t) - A_\infty \|_{H^k(M_r)} \leq C \, \delta_1 \, e^{-(t-\tau)/C_1}  \left( 1 + \sup_{\tau \leq s} \| A(s) \|^{k}_{H^{k-1}(M_{r/2})} \right).
\end{equation}
Assume the desired bound for $k-1.$ Over $M_{r/2},$ we have
\begin{equation*}
\begin{split}
\| A(s) \|^{k}_{H^{k-1}} & \leq \left( \|A(\tau)\|_{H^{k-1}} + \|A(s) - A(\tau)\|_{H^{k-1}} \right)^k \\
& \leq \left( \|A(\tau)\|_{H^{k-1}} + C \delta_1 \left( 1 + \|A(\tau)\|^{(k-1)!}_{H^{k-2}} \right) \right)^k\\
& \leq C \left( 1 + \|A(\tau)\|_{H^{k-1}}^{k!} \right)
\end{split}
\end{equation*}
for $s \geq \tau,$ since $\delta_1 < \bar{\delta}_1.$ Substituting into (\ref{convbootstrap}) gives the $k$'th Sobolev bound on $M_r.$

Note that Theorem \ref{selfdual} and (\ref{kappa}) also imply that the curvature does not concentrate anywhere on $M$ as $t\to \infty.$ The above estimates, with $\{x_i \} = \emptyset,$ imply smooth convergence on all of $M.$ 
\end{proof}

\begin{cor}\label{convergencecorollary} Assume that $\|F^+\| < \delta,$ and there exists an Uhlenbeck limit $(E_\infty, A_\infty)$ which is an instanton with $H^{2+}_{A_\infty}=0.$ Then $E_\infty = E,$ and the flow converges smoothly to a connection gauge-equivalent to $A_\infty.$
\end{cor}
\begin{proof} By assumption, there exist times $t_i \to \infty$ for which (\ref{convergencenonconc}) and (\ref{closeness}) are satisfied.
Since $F^+_{A_\infty}=0$ and $\| F^+(t)\|_{L^\infty(M)}$ is uniformly bounded, (\ref{convergencelowselfdual}) is also satisfied for $\tau = t_i$ sufficiently large.
Therefore, smooth convergence $A(t) \to A(\infty)$ on $M$ 
follows from Theorem \ref{convergence}.
The limit, $A(\infty),$ is related to $A_\infty$ by a bundle map $u = \lim u_i$ defined on $M_0.$ Both connections are smooth, and $u \in L^\infty(M)$ by definition. 
Bootstrapping via (\ref{conntransf}) shows that $u$ is itself smooth over $\{x_i\},$ and $u(A(\infty)) = A_\infty$ on $M.$ 
\end{proof}


\begin{rmk} The results of this section are strongly analogous to those of Topping \cite{toppingrigidity} on harmonic map flow between spheres. He concludes uniqueness of the positions of the bubbles, whereas we show that they do not form. See \cite{thesis} for further discussion of the contrast between harmonic map and Yang-Mills flow, also explored by Grotowski and Shatah \cite{groshat} in the equivariant setting.
\end{rmk}

\section{Global behavior near the minimum energy}\label{asdsection}

In this final section, we derive several consequences of the above results for the global behavior of the Yang-Mills flow at low self-dual energy. The assumptions and notation remain as in Section \ref{convergencesection}.

We begin with a trivial gauge-fixing lemma which is convenient for controlling the flow at short time, for non-simply-connected $M.$ The proof expresses the fact that two flat connections which are $L^p$ close, $p \geq 1,$ are close modulo gauge in any norm.

\begin{lemma}\label{gaugelemma} Let $0 < r < R_0, 1 \leq p \leq q \leq \infty.$ Assume that $A_1, \bar{A}_1, A_2, \bar{A}_2,$ are smooth connections on $E$ over $\Omega \subset M,$ with $\bar{A}_1, \bar{A}_2$ flat. 
There exists $\bar{\epsilon} > 0$ as follows.

If, for $\epsilon < \bar{\epsilon},$ both
\begin{equation}\label{gaugelemmacloseness1}
\| A_1 - A_2 \|_{L^p(\Omega)} \leq \epsilon
\end{equation}
and
\begin{equation}\label{gaugelemmacloseness2}
\| A_1 - \bar{A}_1 \|_{L^q(\Omega)} + \| A_2 - \bar{A}_2 \|_{L^q(\Omega)} \leq  \epsilon
\end{equation}
then there exists a gauge transformation $u$ over $\Omega_r$ such that
$$\|u(A_1) - A_2\|_{L^q(\Omega_r)} \leq C_{\ref{gaugelemma}} \epsilon.$$

The constant and $\bar{\epsilon}$ depend on $\Omega, r, p, q,$ and the structure group of $E.$
\end{lemma}
\begin{proof} If $\Omega$ is simply-connected, then there exist gauge transformations $u_1, u_2$ such that $u_1(\bar{A}_1) = u_2(\bar{A}_2) = 0.$ The triangle inequality, applied to (\ref{gaugelemmacloseness2}), implies
$$\|u_1(A_1) - u_2(A_2)\|_{L^q(\Omega)} \leq 2 \epsilon.$$
Choosing $u = u_2^{-1} u_1,$ we obtain the desired result.


In general, (\ref{gaugelemmacloseness1}), (\ref{gaugelemmacloseness2}), and the triangle inequality imply that
$\|\bar{A}_1 - \bar{A}_2\|_{L^p(\Omega)} \leq C \epsilon.$
Fix a finite cover $\{B^a \}_{a = 1}^N$ of $\Omega_r,$ consisting of geodesic balls of radius $r/2$ centered in $\Omega_r.$ Assume, without loss of generality, that in the local gauge on each $B^a,$ the connection matrix $\bar{A}^a_2 \equiv 0.$
Then, since $p \geq 1$ and $\|\bar{A}^a_1\|_{L^p(B^a)} \leq C \epsilon$ for all $a,$ we may choose points $x^a \in B^a$ with the following property. For the minimizing geodesic $\gamma$ between $x^a$ and $x^b$ inside $B^a \cup B^b,$ there holds 
\begin{equation}\label{radialgaugecloseness}
\int_{\gamma \cap B^a} |\bar{A}^a_1(\gamma(t)) | |\gamma'| dt \leq C \epsilon
\end{equation}
and similarly for $\bar{A}^b_1$ over $B^b.$
Now let $u^a$ be the radial gauge for $\bar{A}^a_1$ 
on $B^a$ centered at $x^a,$ with $u^a(x^a) = 1,$ and put
$$w^{ab} = \left. u^b \right|_{B^a} (u^a )^{-1} .$$
Then by (\ref{radialgaugecloseness}), for $\epsilon \leq \bar{\epsilon},$ we have
\begin{equation}\label{smalltransitions}
\left| \log w^{ab} \right| + \left| d w^{ab} \right| \leq C \epsilon.
\end{equation} 
over $B^a \cap B^b.$ In fact, the $w^{ab}$ are initially constant; 
we then modify the local frames $u^a$ to obtain a frame $u$ over $\Omega_r.$ Write 
$$U^c = \bigcup_{a=1}^c B^a.$$
As a base case, on $U^1 = B^1,$ let $u = u^1,$ so that $u(\bar{A}_1) \equiv 0.$ 
 Assume, by way of induction, that (\ref{smalltransitions}) holds for all $a,b;$ and moreover that for $ a, b < c,$ $u^a$ and $u^b$ agree on $B^a \cap B^b$ and hence define a frame $u$ on $U^{c-1}$ such that
\begin{equation}\label{desiredgaugecloseness}
\| u(\bar{A}_1) - \bar{A}_2 \|_{L^\infty(U^{c-1})} \leq C \epsilon.
\end{equation} 
Since $B^c \cap U^{c-1}$ has finitely many connected components, one can easily write down a gauge $w$ on $B^c$ with $w = w^{ac}$ on $B^a \cap B^c$ for all $a < c,$ and $\left| dw \right| < C \epsilon.$ We then modify
$$u^c\to w \cdot u^c.$$
The result agrees with $u$ on the overlap, and again satisfies (\ref{smalltransitions}). Since the balls $B^a$ are fixed and of finite number, for $\epsilon < \bar{\epsilon}$ the induction hypotheses (\ref{smalltransitions}), (\ref{desiredgaugecloseness}) will be satisfied appropriately.

We have now constructed 
a gauge transformation $u$ such that $\| u(\bar{A}_1) - \bar{A}_2 \|_{L^\infty(\Omega_r)} \leq C \epsilon.$ 
Over $\Omega_r,$ the triangle inequality gives
$$\|u(A_1) - A_2\|_{L^q} \leq \|u(A_1) - u(\bar{A}_1)\|_{L^q} + \|u(\bar{A}_1) - \bar{A}_2 \|_{L^q} + \|\bar{A}_2 - A_2\|_{L^q} \leq C \epsilon$$
as desired. \end{proof}

\begin{thm}\label{parataubes} \emph{(Theorem 1.1 of Taubes \cite{taubes}, parabolic version.)} 
Let $(\bar{E}, \bar{A})$ be a flat bundle on $M$ with $H_{\bar{A}}^{2+}=0,$ and let $r$ and $\epsilon$ be as in Lemma \ref{excision}. For any $p > 1, K^+ > 0,$ and points $x_1, \ldots, x_m \in M,$ there exist $\delta_1, \epsilon_1> 0$ as follows. If $A(0)$ is a connection on $E$ with
\begin{equation*}
\begin{split}
\|F_{A(0)}^+\|_{L^{2p}(M)} & \leq K^+ \\
\|F_{A(0)}^+\|_{L^2(M)} & \leq \delta_1 \\
\end{split}
\end{equation*}
\begin{equation}\label{taubescloseness}
\| u(A(0)) - \bar{A} \|_{H^1(M_{r/2})} \leq \epsilon_1
\end{equation} 
then the flow with initial data $A(0)$ converges smoothly as $t \to \infty$ to an instanton near $\bar{A}$ in $L^4(M_r)$ modulo gauge.\footnote{In fact it will be close over $M_r$ in any Sobolev norm for $\delta_1 + \epsilon_1$ sufficiently small, as remarked in \cite{donkron} \S 4.4.2.}
\end{thm}
\begin{proof} Assuming $C_{(\ref{localsobolev})} \, \epsilon_1 < \epsilon,$ 
(\ref{taubescloseness}) and the local Sobolev inequality (\ref{localsobolev}) imply both a Poincar\'e estimate (\ref{poincare}) for $A(0),$ and an initial  
energy bound $$\|F(0)\|_{L^2(M_{3r/2})} < C \epsilon_1.$$

To control the curvature for a short time, we apply the estimates of Section \ref{curvestimates}, with $u(t) = |F^+(t)|^2.$ The proof of Proposition \ref{premoser}, as given for the $L^2$ norm, can with heavier notation be 
adapted to the $L^p$ norm, for any $p > 1,$ to give a uniform bound
$$\| u(t)\|_{L^{p}} \leq K^+ + C.$$
Proposition \ref{moser}, with the explicit constant, then implies $\| u(t) \|_{L^\infty} \leq C \, \left( 1 + t^{-2/p} \right),$ i.e.
\begin{equation*}
\begin{split}
\| F^+(t) \|_{L^\infty} & \leq C \, \left( 1 + t^{-1/p} \right).
\end{split}
\end{equation*}
Since $p > 1,$ the function $\| F^+(t) \|_{L^\infty}$ is $L^1,$ 
and there exists $0 < \tau < 1$ such that
$$\int_0^\tau \|F^+(t)\|_{L^\infty(M)} \, dt < \epsilon_1.$$
Theorem \ref{selfdual} then implies
$$\|F(t)\|_{L^2 (M_{r} )} < C \epsilon_1$$
for $0 \leq t \leq \tau.$

Assume first that $M$ is simply-connected. 
From Proposition \ref{dstarf} and the energy inequality, the curvature at time $\tau$ on $M_{3r/2}$ is bounded by $C(\delta_1 + \epsilon_1),$ as are all its derivatives. We therefore have the curvature-dependent bounds sufficient for the Coulomb gauge patching argument of \cite{donkron}, \S 4.4.2.\footnote{Proposition \ref{dstarf} exactly replaces \cite{donkron}, Theorem 2.3.8, which gives curvature-dependent bounds from the ASD equation.} In particular, Proposition 4.4.10 of \cite{donkron} states that for $\delta_1$ sufficiently small, there exists a gauge transformation $u$ on $M_{r}$ with 
$$\|u( A(\tau)) - \bar{A} \|_{L^4(M_{r})} < C \epsilon_1.$$ 
The existence of $\delta_1>0$ as required now follows from Theorem \ref{convergence}.

If $M$ is not simply-connected, we argue as follows. Let $\pi : \tilde{M} \to M$ be the universal cover, and choose a strongly 
simply-connected domain $\Omega \subset \tilde{M}$ covering $M_r,$ which is a finite union of preimages of $B^a \subset M_{r/2},$ with $B^a \cap B^b$ connected.\footnote{This can be done for instance by lifting the $B^a$ to $\tilde{M}$ using a set of 
paths which form a spanning tree for the incidence graph of $\{B^a\}.$} 
By \cite{donkron}, Prop. 4.4.10, as before we may choose a gauge $v$ on $\Omega_{r/2}$ such that
\begin{equation}
\|v( \pi^*A(\tau)) - \pi^*\bar{A}\|_{L^4(\Omega_{r/2})} < C \epsilon_1.
\end{equation}
If this is done using preimages of Coulomb gauges on the $ B^a,$ then $v^{-1}(\pi^*\bar{A})$ descends to
a flat connection on $E$ over $M_{r/2}.$ 
Calling this connection $\bar{A}_1,$ we obtain
\begin{equation}\label{a2closeness}
\|A(\tau) - \bar{A}_1\|_{L^4(M_{r/2})} \leq C \epsilon_1.
\end{equation}
Note also, directly from the flow equation and the energy inequality, that
\begin{equation}\label{l2closeness}
\| A(\tau) - A(0) \|_{L^2(M)} \leq \tau^{1/2} \left( \int_0^{\tau} \|D^*F\|_{L^2(M)}^2 \, dt \right)^{1/2} \leq \delta_1.
\end{equation} 

In view of (\ref{taubescloseness}), (\ref{a2closeness}), and (\ref{l2closeness}), we may apply Lemma \ref{gaugelemma} with $p=2$ and $q=4$ to $A(0), A(\tau),$ and the two flat connections $\bar{A}, \bar{A}_1.$ We conclude that there in fact exists a gauge transformation $u$ on $M_{r}$ with
$$\| u(A(\tau) ) - A(0)\|_{L^4(M_{r})} \leq C_{\ref{gaugelemma}} (\delta_1 + \epsilon_1).$$
The desired result follows again from Theorem \ref{convergence}.
\end{proof}

\begin{thm}\label{chargeone} Assume $E$ has structure group $SU(2)$ with $\kappa (E) = 1,$ $H^{2+}(M)=0,$ and $\pi_1(M)$ has no nontrivial representations in $SU(2).$ There exists $\delta_1 > 0$ such that if $\|F^+(0)\| < \delta_1,$ then no bubbling occurs as $t\to \infty.$ If an Uhlenbeck limit $A_\infty$ is an instanton with $H^{2+}_{A_\infty} = 0,$ then it is unique, and the flow converges exponentially.
\end{thm}
\begin{proof} Assume, by way of contradiction, that bubbling occurs on a sequence $t_i \to \infty.$ The 
blowup limits constructed by Schlatter 
\cite{schlatterlongtime} at a presumed singularity, as well as the Uhlenbeck limit, preserve the structure group. 
Due to the $L^\infty$ bound on $F^+,$ the blowup limit at a bubble must be anti-self-dual. Since it is a nontrivial $SU(2)$ instanton on $S^4$ 
with energy less than $4 \pi^2 + \delta_1^2,$ it must have energy exactly $4 \pi^2.$

Let $A_\infty$ be the Uhlenbeck limit obtained from Theorem \ref{uhllimexist} on the same sequence $\{ t_i \},$ which has total energy less than $\delta_1^2.$ 
We may apply \cite{donkron}, Prop. 4.4.10, on the universal cover of $M$ as in the preceding proof (with $\{x_i \} = \emptyset$). 
This implies that for $\delta_1$ sufficiently small, $A_\infty$ must be $L^4$-close to a flat connection $\bar{A}$ on $M,$ modulo gauge. But, by the assumption on $\pi_1(M),$ $\bar{A}$ is equivalent to the product connection on the trivial bundle, and so 
$$H^{2+}_{\bar{A}} = H^{2+}(M) = 0.$$ 
Hence, for $\delta_1$ sufficiently small, $A_\infty$ satisfies a Poincar{\'e} estimate (\ref{poincare}).
By Theorem \ref{uhllimexist}, $A_\infty$ is Yang-Mills, hence by (\ref{selfdualbianchi}) and (\ref{poincare}) it must be anti-self-dual. By integrality of $\kappa(E)$ 
for $SU(2)$-bundles, $A_\infty$ is flat, and also equivalent to the product connection. Corollary \ref{convergencecorollary} then implies that the flow converges, which is a contradiction.

Therefore any Uhlenbeck limit exists smoothly.
The last statement follows again from Corollary \ref{convergencecorollary}.
\end{proof}

\begin{thm}\label{asympstab} The instantons with $H^{2+} = 0$ are asymptotically stable in the $H^1$ topology. In other words, given an anti-self-dual connection $A_1$ 
with $H^{2+}_{A_1} = 0$ and an $H^1$ neighborhood $U$ of $A_1,$ there exists a neighborhood $U_1 \subset U$ of initial connections for which the limit under the flow will again be an instanton with $H^{2+} = 0,$ lying in $U$ modulo smooth gauge transformations.
\end{thm}
\begin{proof} Choose the given instanton $A_1$ for the required smooth connection $D_1 = D_{ref} + A_1$ in Struwe's construction---see Section \ref{shorttimesection}. Let $\epsilon > 0$ be such that the $H^1$ neighborhood of $A_1$ of size $C_{\ref{localexistence}} \epsilon$ is contained in $U,$ and choose $U_1$ to be the neighborhood of size $\epsilon.$ By Theorem \ref{localexistence}, the gauge-equivalent flow (\ref{gaugequiv}) with initial data in $U_1$ will remain in $U$ for a time $\tau.$
Choosing $U_1$ smaller, we also obtain the three conditions (\ref{convergencelowselfdual}), (\ref{convergencenonconc}), (\ref{closeness}) with $\delta_1$ arbitrarily small. We are then in the situation of Theorem \ref{convergence}, which may be applied with $\tau_0 = \tau$ and $\{x_i\} = \emptyset.$
\end{proof}

\begin{cor} On any $SU(2)$-bundle $E \to M,$ there exists a nonempty $H^1$-open set of initial connections for which the Yang-Mills flow exists for all time and, if $\kappa(E) \geq 0$ 
and $H^{2+}(M)=0,$ converges exponentially.
\end{cor}
\begin{proof} Following Freed and Uhlenbeck \cite{freeduhl}, for any $\delta_1,$ one can construct smooth pointlike 
 $SU(2)$-connections with $\|F^+\|_{L^2} < \delta_1$ and $\|F^+\|_{L^\infty} < C$ (p. 124). Provided $H^{2+}(M)=0,$  for $\delta_1$ sufficiently small, Theorem \ref{parataubes} yields exponential convergence to an instanton $A_\infty.$ By Theorem \ref{asympstab}, convergence holds for initial data in an $H^1$-open neighborhood.
\end{proof}

\begin{cor}\label{retraction} There exists a $\scrG_E$-invariant $H^1$-open set $N \subset \scrA_E,$ such that the flow gives a deformation retraction, with respect to the $H^k$ topology on $\scrA_E / \scrG_E,$ for $k >> 1,$ from $N \cap H^k$ onto the instantons with $H^{2+}=0.$
\end{cor}
\begin{proof} Let $N$ be the union of the neighborhoods obtained in Theorem \ref{asympstab}, which we may take to be invariant under smooth gauge transformations. By Kozono et. al., Theorem C, for $k$ sufficiently large, the Yang-Mills flow is in fact unique modulo gauge as long as it exists in $H^k,$ and therefore gives a well-defined map on $N / \scrG_E$ fixing the instantons.

It remains to show that this map is continuous modulo gauge in the $H^k$ topology. 
This follows by combining the parabolic theory of Section \ref{shorttimesection} with the uniform convergence statement of Theorem \ref{convergence}. 
For, two connections in $N$ which are initially $H^k$-close, for $k \geq k_0,$ 
 remain so under the gauge-equivalent flow (\ref{gaugequiv}). They therefore remain close, modulo gauge, for an arbitrarily long time; but then both are close to their respective limits under the Yang-Mills flow.

To spell this out, first note that since the flow converges smoothly, we always have
\begin{equation}\label{crudebound}
\| A(t) \|_{H^k} < K
\end{equation}
for a constant $K$ depending on $A(0) \in N.$ 
 By Section \ref{shorttimesection}, there exists $\tau> 0$ such that the map
$$A(t) \to A(t + \tau)$$
is continuous in $H^k,$ modulo gauge, for all solutions 
satisfying (\ref{crudebound}). Since continuity is a local property and closed under composition, we conclude that for any $0 \leq T < \infty,$ the map $A(t) \to A(t + T)$ is continuous modulo gauge on $N \times \LB 0 , \infty \right).$ 

Now, let $A_1$ be an instanton with $H_{A_1}^{2+}=0,$ and fix a neighborhood $U \ni A_1.$ 
By Theorem \ref{selfdual}, there exists a sufficiently small neighborhood $V \subset U$ with the property that if $A(t)$ satisfies
\begin{equation}\label{v}
A(\tau - 1), A(\tau) \in V
\end{equation}
for any $\tau \geq 1,$ then
\begin{equation}\label{desiredlimit}
\lim_{t \to \infty} A(t) \in U.
\end{equation}
To show continuity at infinite time, let $A'(t)$ 
be a solution with $A'(0) \in N$ and $\lim_{t \to \infty} A'(t) = A_1.$ Choose $\tau$ such that $A'(t)$ satisfies (\ref{v}). Then by continuity on $\LB 0, \tau \RB,$ for $A''(0)$ sufficiently close to $A'(0),$ $A''(t)$ will also satisfy (\ref{v}) after changing gauge. Therefore $A''(t)$ also satisfies (\ref{desiredlimit}), as desired.
\end{proof}

\subsection*{Acknowledgements.} These results formed part of the author's PhD thesis \cite{thesis} at Columbia University, and he thanks his advisor, Panagiota Daskalopoulos, for vital direction and support. Thanks also to Richard Hamilton for noticing a simplification, to Michael Struwe for an encouraging discussion, and to D. H. Phong for initially suggesting the problem.


\begin{thebibliography}{6}

\bibitem{atiyahbott} M. F. Atiyah and R. Bott. The Yang-Mills equations over Riemann surfaces. Philosophical Transactions of the Royal Society of London. Series A, Mathematical and Physical Sciences (1983), 523-615.

\bibitem{bl} J.-P. Bourguignon, H. B. Lawson. Stability and isolation phenomena for Yang-Mills fields. Comm. Math. Phys. 79 (1981), no. 2, 189–230.

\bibitem{cdy} K. C. Chang, W. Y. Ding, and R. Ye. Finite-time blowup of harmonic maps from surfaces. J. Diff. Geom. 36 (1992), 507-515.

\bibitem{chenshen} Y. Chen. and C.-L. Shen. Monotonicity formula and small action regularity for Yang-Mills flows in higher dimensions. Calc. Var. 2 (1994), 389-403.

\bibitem{dask} G. Daskalopoulos. The topology of the space of stable bundles on a compact Riemann surface. J. Differential Geom. 36 (1992), no. 3, 699-746.

\bibitem{daskwent04} G. Daskalopoulos and R. Wentworth. Convergence properties of the Yang-Mills flow on Kahler surfaces. J. Reine Angew. Math. 575 (2004), 69-99.

\bibitem{daskwent07} ---. On the blow-up set of the Yang-Mills flow on Kahler surfaces. Math. Z. 256 (2007), no. 2, 301-310.

\bibitem{don4d} S. K. Donaldson. An application of gauge theory to four dimensional topology. J. Differential Geom. 18 (1983), 279-315.

\bibitem{donsurface} ---. Anti self-dual Yang-Mills connections over complex algebraic surfaces and stable vector bundles. Proceedings of the London Mathematical Society 50.1 (1985), 1-26.

\bibitem{donkron} S. K. Donaldson and P. B. Kronheimer. The Geometry of Four-Manifolds. Oxford University Press (1990).

\bibitem{freeduhl} D. Freed and K. Uhlenbeck. Instantons and Four-Manifolds. MSRI Research Publications I (1984).

\bibitem{gt} D. Gilbarg, and N. S. Trudinger. Elliptic Partial Differential Equations of Second Order. Springer (1983).

\bibitem{groshat} J. Grotowski and J. Shatah. Geometric evolution equations in critical dimensions. Calc. Var. 30 (2007), 499-512.

\bibitem{hamilton} R. S. Hamilton. Monotonicity formulas for parabolic flows on manifolds. Comm. Anal. Geom 1.1 (1993), 127-137.

\bibitem{hongtian} M. C. Hong and G. Tian. Asymptotical behavior of the Yang-Mills flow and singular Yang-Mills connections. Math. Ann. 330 (2004), 441-472.

\bibitem{kozono} H. Kozono, Y. Maeda, H. Naito. Global solution for the Yang-Mills gradient flow on 4-manifolds. Nagoya Math. J. 139 (1995), 93–128.

\bibitem{lawson} H. B. Lawson. The Theory of Gauge Fields in Four Dimensions. Regional Conference Series in Mathematics, Number 58, AMS (1985).

\bibitem{li} P. Li. Geometric Analysis. Cambridge University Press (2012).

\bibitem{lm} J. L. Lions, and E. Magenes. Non-Homogeneous Boundary Value Problems and Applications. Springer (1973).

\bibitem{naito} H. Naito. Finite time blowing-up for Yang-Mills gradient flow in higher dimensions. Hokkaido Math. J. 23 (1994), no. 3, 451-464.

\bibitem{rade} J. Rade. On the Yang-Mills heat equation in two and three dimensions. J. Reine Angew. Math. 120 (1998), 117-128.

\bibitem{schlatterlongtime} A. E. Schlatter. Long-time behavior of the Yang-Mills flow in four dimensions. Ann. Global Anal. Geom. 15 (1997), no. 1, 1-25.

\bibitem{schlatterglobal} ---. Global existence of the Yang-Mills flow in four dimensions. J. Reine Angew. Math. 479 (1996), 133-148.

\bibitem{schoen} R. Schoen. Analytic aspects of the harmonic map problem. Seminar on nonlinear partial differential equations, Springer (1984), 321-358.


\bibitem{sstz} A. E. Schlatter, M. Struwe and A. S. Tahvildar-Zadeh. Global existence of the equivariant Yang-Mills heat flow in four space dimensions. Am. J. Math. 120 (1998), 117-128.

\bibitem{sedlacek} S. Sedlacek. The Yang-Mills functional over four-manifolds. Comm. Math. Phys. 86 (1982), 515-527.

\bibitem{siu} Y. T. Siu. Lectures on Hermitian-Einstein Metrics for Stable Bundles and Kahler-Einstein Metrics. Birkhauser, 1986.

\bibitem{struwehm} M. Struwe. On the evolution of harmonic mappings of Riemann surfaces. Comment. Math. Helv. 60 (1985), no. 4, 558-581.

\bibitem{struwe} ---. The Yang-Mills flow in four dimensions. Calc. Var. 2 (1994), 123-150.

\bibitem{taubes} C. H. Taubes. Self-dual Yang-Mills connections on non-self-dual {4}-manifolds. J. Diff. Geom. 17.1 (1982), 139-170.

\bibitem{toppingrigidity} P. M. Topping. Rigidity in the harmonic map heat flow. J. Diff. Geom. 45.3 (1997), 593-610.


\bibitem{uhlenbecklp} K. Uhlenbeck. Connections with $L^p$ bounds on curvature, Comm. Math. Phys. 83 (1982), 31-42.

\bibitem{uhlenbeckremov} K. Uhlenbeck. Removable singularities in Yang-Mills fields, Comm. Math. Phys. 83 (1982), 11-30.

\bibitem{thesis} A. Waldron. Self-duality and singularities in the Yang-Mills flow. Ph.D thesis, Columbia University, 2014.

\bibitem{weinkove} B. Weinkove. Singularity formation in the Yang-Mills flow. Calc. Var. 19 (2004), 221-220.



\end{thebibliography}
\end{document}